\documentclass[12pt]{article}
\usepackage{amsmath,amssymb,amscd,amsthm,verbatim}
\usepackage{enumerate}
\usepackage{microtype}
\usepackage{url}
\usepackage{hyperref}
\usepackage[top=1in,bottom=1in,left=.91in,right=.91in]{geometry}

\title{FI-modules over Noetherian rings}

\author{Thomas Church, Jordan S. Ellenberg, Benson Farb, and Rohit Nagpal \thanks{The first, second, and third authors gratefully acknowledge support from the National Science Foundation. The second author's work was partially supported by a Romnes Faculty Fellowship.}}

\theoremstyle{plain}
\newtheorem{theorem}{Theorem}[section]
\newtheorem{maintheorem}{Theorem}
\newtheorem*{theoremC}{Theorem~\ref{thm:inductive}}
\newtheorem{proposition}[theorem]{Proposition}
\newtheorem{lemma}[theorem]{Lemma}
\newtheorem{corollary}[theorem]{Corollary}
\theoremstyle{definition}
\newtheorem{remark}[theorem]{Remark}

\newtheorem{definition}[theorem]{Definition}

\newcommand{\nc}{\newcommand}
\nc{\dmo}{\DeclareMathOperator}

\nc{\FF}{\mathbf{F}}
\nc{\Q}{\mathbb{Q}}
\nc{\R}{\mathbb{R}}
\nc{\Z}{\mathbb{Z}}
\nc{\C}{\mathbb{C}}
\nc{\N}{\mathbb{N}}
\nc{\CC}{\mathbf{C}}

\renewcommand{\O}{\mathcal{O}}
\nc\p{\mathfrak{p}}
\nc\XX{\mathbf{X}}
\renewcommand{\epsilon}{\varepsilon}
\nc\X{X}
\nc\Y{Y}

\dmo\FI{FI}
\dmo\FIMod{FI-Mod}
\dmo\dMod{-Mod}
\nc\RMod{R\dMod}
\dmo\Sets{Sets}
\dmo\Bij{Bij}

\dmo{\GL}{GL}
\dmo{\Aut}{Aut}
\dmo{\Stab}{Stab}
\dmo\SL{SL}
\renewcommand{\sl}{\mathfrak{sl}}
\DeclareMathOperator*{\colim}{colim}
\dmo\im{im}
\dmo\id{id}
\dmo\fd{fd}
\dmo\Sym{Sym}
\dmo\End{End}
\dmo\Conf{Conf}
\dmo\op{op}
\dmo\coker{coker}
\dmo\Inj{Inj}
\dmo\Map{Map}
\dmo\chr{char}
\dmo\Hom{Hom}
\dmo\Ext{Ext}
\dmo\spn{span}
\dmo\Ind{Ind}
\dmo\Res{Res}

\nc{\set}[1]{\{#1\}}
\nc{\beq}{\begin{displaymath}}
\nc{\eeq}{\end{displaymath}}
\nc{\beqn}{\begin{equation}}
\nc{\eeqn}{\end{equation}}
\nc{\tensor}{\otimes}
\renewcommand{\approxeq}{\sim}
\def\HH{\mathcal{H}}
\nc{\tSminus}[1]{\widetilde{S}_{-#1}}
\nc{\Xia}{\Xi_a}

\nc{\disjoint}{\sqcup}
\nc{\astT}[1]{\ensuremath{\ast_{#1}}}

\nc{\ra}{\rightarrow}
\nc{\inj}{\hookrightarrow}
\nc{\surj}{\twoheadrightarrow}

\nc{\bwedge}{\textstyle{\bigwedge}}
\nc{\abs}[1]{\left\lvert#1\right\rvert}
\nc{\coloneq}{\mathrel{\mathop:}\mkern-1.2mu=}

\nc{\margin}[1]{\marginpar{\scriptsize #1}}
\nc{\para}[1]{\medskip\noindent\textbf{#1.}}

\nc{\arXiv}[1]{\href{http://arxiv.org/abs/#1}{\nolinkurl{arXiv:#1}}}
\nc{\arXivV}[2]{\href{http://arxiv.org/abs/#1}{\nolinkurl{arXiv:#1v#2}}}
\nc{\myemail}[1]{\href{mailto:#1}{\nolinkurl{#1}}}

\begin{document}

\maketitle
\begin{abstract}
FI-modules were introduced by the first three authors in \cite{CEF} to encode sequences of representations of symmetric groups. Over a field of characteristic 0, finite generation of an FI-module implies representation stability for the corresponding sequence of $S_n$-representations.
In this paper we prove the Noetherian property for FI-modules over arbitrary Noetherian rings: any sub-FI-module of a finitely-generated FI-module is finitely generated. This lets us extend many of the results of \cite{CEF} to representations in positive characteristic, and even to integral coefficients. We focus on three major applications of the main theorem: on the integral and mod $p$ cohomology of configuration spaces; on diagonal coinvariant algebras in positive characteristic; and on an integral version of Putman's central stability for homology of congruence subgroups.
\end{abstract}

\section{Introduction}

In \cite{CEF}, the first three authors investigated the theory of {\em FI-modules}, which encode sequences of representations of symmetric groups connected by families of linear maps.  The category of FI-modules defined in \cite{CEF} admits a natural notion of {\em finite generation}, which is central to the story told there.  In particular, finitely-generated FI-modules over a field of characteristic $0$ correspond to sequences of representations whose dimensions and  characters behave ``eventually polynomially."  This turns out to be essentially equivalent to the phenomenon that was called ``representation stability" in the earlier work of the first and third authors~\cite{CF}.

In much of \cite{CEF} it was critical that we consider FI-modules over a field of characteristic 0. Most notably, this was used in the proof there that the category of FI-modules over a field of characteristic 0 is {\em Noetherian}; that is, any sub-FI-module of a finitely-generated FI-module is again finitely generated.  This property is essential for many of the applications in \cite{CEF}.  The main purpose of the present paper is to prove the Noetherian property for FI-modules over \emph{arbitrary Noetherian rings} $R$.

This allows us to generalize many of the applications in \cite{CEF} beyond the case of fields of characteristic $0$, and to produce new applications as well.  We discuss three such results in this paper:
\begin{itemize}
\item We prove new theorems about the integral and mod $p$ cohomology of configuration spaces on manifolds, generalizing results in \S 4 of \cite{CEF};
\item We characterize the dimensions of diagonal coinvariant algebras over fields of positive characteristic, generalizing results in \S 3.2 of \cite{CEF};
\item We prove a complement to a recent theorem of Putman \cite{P} on the homology groups of congruence subgroups. Putman shows that the mod $p$ homology of these subgroups satisfies a version of representation stability, with an explicit stable range, for all primes $p$ above a certain explicit threshold.  We prove a similar theorem, which does not provide an explicit range, but which 
holds for coefficients of any characteristic, even when the coefficients are not a field.
\end{itemize}

\subsection{The Noetherian property}
Let $\FI$ be the category whose objects are finite sets and whose morphisms are injections. The category $\FI$ is equivalent to its full subcategory whose objects are the sets $\{1,\ldots, n\}$ as $n$ ranges over natural numbers $n\geq 0$.  For simplicity we denote $\{1,\ldots, n\}$ by $[n]$ hereafter, with $[0]\coloneq \emptyset$.

Let $R$ be a commutative ring.\footnote{The restriction to commutative rings is probably not essential;  see for instance the discussion  of $\FI[G]$-modules by Jimenez Rolland in \cite{RJR}.} An \emph{FI-module over $R$} is a covariant functor $V$ from $\FI$ to the category of $R$-modules.  Given a finite set $S$ we denote the $R$-module $V(S)$ by $V_S$, and in particular we denote $V([n])$ by $V_n$.  Since $\End_{\FI}([n])=S_n$, any FI-module $V$ determines for each $n\geq 0$ an $S_n$-representation $V_n$ (that is, an $R[S_n]$-module). Moreover, the FI-module $V$ determines linear maps $V_m\to V_n$ corresponding to the injections $[m]\inj [n]$.  The FI-module structure imposes no maps from $V_m$ to $V_n$ when $n < m$. The usual notions from the theory of modules, such as submodule and quotient module, carry over to FI-modules. 

The applications in this paper are all based on the notion of finite generation of an FI-module.  An FI-module $V$ is \emph{finitely generated} if there is a finite set $S$ of elements in $\coprod_iV_i$ so that no proper sub-FI-module of $V$ contains $S$.  This condition was put to much use in \cite{CEF}; in particular, over a field of characteristic 0, finite generation of $V$ implies representation stability in the sense of \cite{CF} for the sequence $\{V_n\}$ of $S_n$-representations.

This paper has three main results; all three are proved in \S\ref{section:noetherian} below. 
When $k$ is a field of characteristic $0$, Theorem~\ref{thm:noetherian} was proved earlier in \cite[Theorem~2.3]{Snowden} and \cite[Theorem~2.60]{CEF}, and Theorem~\ref{thm:polynomial} was proved in \cite[Theorem~2.67]{CEF}.

\begin{maintheorem}[{\bf Noetherian property}]
\label{thm:noetherian}
If $V$ is a finitely-generated FI-module over a Noetherian ring $R$, and $W$ is a sub-FI-module of $V$, then $W$ is finitely generated.
\end{maintheorem}

\begin{maintheorem}[{\bf Polynomial dimension}]
\label{thm:polynomial}
Let $k$ be any field, and let $V$ be a finitely-generated FI-module over $k$.  Then there exists an integer-valued polynomial $P(T)\in \Q[T]$ so that for all sufficiently large $n$, \[\dim_k V_n = P(n).\]
\end{maintheorem}

\begin{maintheorem}[{\bf Inductive description}]
\label{thm:inductive}
Let $V$ be a finitely-generated FI-module over a Noetherian ring $R$. Then there exists some $N\geq 0$ such that for all $n\in \N$:
\begin{equation}
\label{eq:colim}
V_n=\ \colim_{\substack{S\subset [n]\\\abs{S}\leq N}}\  V_S
\end{equation}
\end{maintheorem}
We emphasize that the colimit in \eqref{eq:colim} is taken over the \emph{poset} of subsets $S\subset [n]$ satisfying $\abs{S}\leq N$ under inclusion. In particular, the permutations do not play a role in defining the right side of \eqref{eq:colim}. However, $S_n$ does act naturally on the right side, and thus Theorem~\ref{thm:inductive} does determine $V_n$ as an $S_n$-representation.

The condition \eqref{eq:colim} in Theorem~\ref{thm:inductive} can be viewed as a reformulation of Putman's ``central stability'' condition \cite[\S1]{P}. One difference is that  we have formulated it as a global condition on the entire FI-module $V$, while Putman defines central stability as a local condition on the adjacent terms $V_{n-1},V_n,V_{n+1}$ separately for each $n$.
Nevertheless, the notions are equivalent.

\begin{remark} FI-modules were originally introduced in order to study various sequences $\{V_n\}$ of $S_n$-representations arising from algebra, combinatorics, and geometry, about which little explicit information is known.  For instance, one often lacks even a formula for the dimension of $V_n$.

The reason that Theorems~\ref{thm:noetherian}, \ref{thm:polynomial}, and \ref{thm:inductive} are so useful in practice is because many examples arise as sub-FI-modules of FI-modules which are readily seen to be finitely generated. In many cases we know nothing more about them except that they admit such an embedding. Nonetheless, Theorem~\ref{thm:noetherian} and Theorem~\ref{thm:polynomial} tell us their dimensions are eventually polynomial in $n$, and Theorem~\ref{thm:inductive} guarantees that they can be built up inductively from a finite amount of data. 
\end{remark}

\begin{remark} When $\chr k=0$, we proved in  \cite[Theorem~2.67]{CEF} that not only the dimensions but also the characters of $V_n$ are eventually polynomial. In the situation of Theorem~\ref{thm:polynomial}, it is reasonable to expect that when $k$ is a field of positive characteristic the Brauer characters of $V_n$ similarly have polynomial behavior. We do not pursue this question here.
\end{remark}

\begin{remark}
The analogue of Theorem~\ref{thm:noetherian} with $\FI$ replaced by a \emph{finite} category was proved by L\"{u}ck \cite[Lemma 16.10b]{Lueck}. However, his methods cannot be extended to infinite categories such as $\FI$.
\end{remark}

The category of FI-modules over a commutative ring $R$ naturally forms an abelian category \cite[\S2.1]{CEF}. As a consequence of Theorem~\ref{thm:noetherian}, the same is true if we restrict to finitely-generated FI-modules.
\begin{corollary}
If $R$ is a Noetherian ring, the category of finitely-generated FI-modules over $R$ is an abelian category.
\end{corollary}
When $R=\C$, this property has already been exploited in Sam--Snowden \cite{SS}, where the abelian category of finitely-generated FI-modules over $\C$ is studied extensively.

\subsection{Applications}

Theorems~\ref{thm:noetherian}, \ref{thm:polynomial}, and \ref{thm:inductive} can be applied to a variety of examples.  In this paper we concentrate on three important examples of FI-modules from algebra, topology and combinatorics. We will prove that each is a finitely-generated FI-module.

As a notational convention, we prepend ``FI'' to the name of a category to denote the category of functors from FI to that category; so an \emph{FI-group} is a functor from FI to the category of groups, an \emph{FI-simplicial complex} is a functor from FI to simplicial complexes, and so forth. Similarly, a \emph{co-FI-space} is a functor from $\FI^{\op}$ to the category of topological spaces, and so on.

\para{Application 1: Congruence subgroups}
Let $R$ be a commutative ring and let $\GL_n(R)$ be the group of automorphisms of $R^n$.  
We can regard $\GL_\bullet(R)$ as an FI-group, where an inclusion $f\colon [n]\hookrightarrow [m]$ induces the homomorphism $f_*\colon \GL_n(R)\to \GL_m(R)$ defined by
\begin{equation}
\label{eq:GLFIgroup}
(f_*M)_{ij}=\begin{cases}
M_{ab}&i=f(a),\ j=f(b) \\
\delta_{ij}&\{i,j\}\not\subset f([n])
\end{cases}
\end{equation}

For any ideal $\p\subset R$, the \emph{congruence subgroup} $\Gamma_n(\p)$ is the kernel of the natural reduction map $\GL_n(R)\to \GL_n(R/\p)$; in other words, $\Gamma_n(\p)$ consists of those invertible matrices that are congruent to the identity matrix modulo $\p$. The map \eqref{eq:GLFIgroup} satisfies $f_*(\Gamma_n(\p))\subset \Gamma_m(\p)$, so these congruence subgroups also define an FI-group $\Gamma_\bullet(\p)$. In particular, for any coefficient ring $A$ and any $m\geq 0$, the homology groups $H_m(\Gamma_n(\p);A)$ form an FI-module $\HH_m(\Gamma_\bullet(\p);A)$ over $A$.

It is known for a wide class of rings $R$ that $\GL_n(R)$ satisfies \emph{homological stability}; that is, $H_m(\GL_n(R);A)\approx H_m(\GL_{n+1}(R);A)$ for $n\gg m$. The corresponding statement is false for $\Gamma_n(\p)$, whose homology with certain coefficient modules grows as $n\to \infty$; this phenomenon is identified as accounting for the ``failure of excision in $K$-theory'' by Charney~\cite{Charney}. However, the striking results of Putman \cite{P} show that in many cases the FI-module $\HH_m(\Gamma_\bullet(\p);A)$ is nevertheless finitely generated. Our results on $\HH_m(\Gamma_\bullet(\p);A)$ complement, and were inspired by, the results of Putman in \cite{P}.

\begin{maintheorem}
\label{thm:congruencefg}
Let $K$ be a number field, let $\O_K$ be its ring of integers, and let $\p\subsetneq \O_K$ be a proper ideal. Fix $m\geq 0$ and a Noetherian ring $A$. Then the FI-module $\HH_m(\Gamma_\bullet(\p);A)$ is finitely generated.
\end{maintheorem}

The following two theorems are immediate corollaries of Theorem~\ref{thm:congruencefg}, by applying Theorem~\ref{thm:polynomial} and Theorem~\ref{thm:inductive} respectively.
\begin{theorem}[{\bf Betti numbers of congruence subgroups}]
\label{thm:congruencepoly}
Let $K$ be a number field with ring of integers $\O_K$, and fix a proper ideal $\p\subsetneq \O_K$. For any $m\geq 0$ and any field $k$, there exists a polynomial $P(T)=P_{\p,m,k}(T)\in\Q[T]$ so that 
for all sufficiently large $n$, 
\[\dim_k H_m(\Gamma_n(\p);k) = P(n).\]
\end{theorem}

\begin{theorem}[{\bf An inductive description of $H_m(\Gamma_n(\p);\Z)$}]
\label{thm:congruenceinductive}
Let $K$ be a number field with ring of integers $\O_K$, and fix a proper ideal $\p\subsetneq \O_K$. For any $m\geq 0$, there exists $N=N_{\p,m}\geq 0$ such that for all $n$:
\[H_m(\Gamma_n(\p);\Z)=\, \colim_{\substack{S\subset [n]\\\abs{S}\leq N}}\, H_m(\Gamma_S(\p);\Z).\]
\end{theorem}

Under the hypothesis that the characteristic of the coefficient field $k$ is either $0$ or at least $9\cdot 2^{m-1} - 3$, Putman proved that Theorem~\ref{thm:congruencepoly} and a version of Theorem~\ref{thm:congruenceinductive} hold for all $n\geq 9\cdot 2^{m}-7$ \cite[Theorems~B and D]{P}.
One of the key tools used by Putman is the representation theory of the symmetric groups, especially the parallels between representations over fields of characteristic 0 and over fields of positive characteristic.  It is the use of this theory that requires the exclusion of fields $k$ of small characteristic. The structural analysis of FI-modules behind Theorem~\ref{thm:noetherian} can be regarded as studying the ``stable representation theory of $S_n$ over $\Z$'', at least to such a degree as this is possible.

\begin{remark}
The restriction to number rings $\O_K$ was not present in \cite{P}, where the corresponding theorem was proved for arbitrary commutative Noetherian rings of finite Krull dimension. But for us this restriction is essential. The reason is that we need to know \emph{a priori} that $H_m(\Gamma_n(\p);k)$ is a finitely-generated $k$-module for all $m\geq 0$ and $n\geq 0$.

For finite-index subgroups of $\GL_n(\O_K)$ such as $\Gamma_n(\p)$, this is guaranteed by the existence of the Borel--Serre compactification. For more general rings it is simply false; for example, for $R=\C[T]$ and $\p=(T)\subset R$, the first homology $H_1(\Gamma_n(\p);\Z)$ surjects to $\sl_n(\p/\p^2)=\sl_n \C$, which is definitely not a finitely-generated abelian group. See the proof of Theorem~\ref{thm:congruencefg} in Section~\ref{sec:congruence} for more details on how this assumption is used.
\end{remark}

\begin{remark}
Theorem~\ref{thm:noetherian} allows us to extend Putman's results to 
coefficients in an abitrary Noetherian ring; in particular, this confirms the conjecture in \cite{P} that the restriction on characteristic is unnecessary.  But there is a cost\,---\,the argument presented here does not provide an explicit stable range, as Putman's does, so that neither theorem implies the other. Furthermore, the methods in this paper only apply to number rings. We remove these shortcomings, while maintaining Putman's exponential stable range, in the forthcoming paper \cite{CE}. 
\end{remark}

\begin{remark}
Calegari~\cite{Ca} has recently determined the rate of growth of the mod-$p$ Betti numbers of level-$p^k$ congruence subgroup of $SL_n(\O_K)$. For example, for $p>3$ and the congruence subgroup $\Gamma_n(p^k)\subset \SL_n(\Z)$, he proves \cite[Lemma~3.5]{Ca} that
\[\dim_{\FF_p}H_m(\Gamma_n(p^k);\FF_p)=\binom{n^2-1}{m}+O(n^{2m-4}).\]
This result complements Theorem~\ref{thm:congruencepoly}: we show that the dimension is exactly some polynomial in $n$ (for
 large enough $n$), while Calegari's result gives the degree of this polynomial and its leading terms.
For other number rings $\O_K$ of degree $[K:\Q]=d$, he obtains a similar estimate (subject to some assumptions on how $p$ splits in $\O_K$) for $\Gamma_n(p^k)\subset \SL_n(\O_K)$ in \cite[Remark~3.6]{Ca}:\[\dim_{\FF_p}H_m(\Gamma_n(p^k);\FF_p)=\frac{n^{2md}}{m!}+O(n^{2d(m-1)}).\] 
\end{remark}

\para{Application 2: Configuration spaces}
Let $M$ be any connected, oriented manifold. 
For any finite set $S$, let $\Conf_S(M)$ denote the space $\Inj(S,M)$ of injections $S\hookrightarrow M$. An inclusion $f\colon S\hookrightarrow T$ induces a restriction map $f^*\colon \Conf_T(M)\to \Conf_S(M)$; this is nothing more than the composition of injections $\Inj(S,T)\times \Inj(T,M)\to \Inj(S,M)$. We can therefore regard $\Conf(M)$ as a \emph{co-FI-space}, i.e.\ a contravariant functor from FI to topological spaces.

When $S=[n]$, the space of injections $[n]\hookrightarrow M$ can be identified with the classical configuration space $\Conf_n(M)$ of ordered $n$-tuples of distinct points in $M$: \[\Conf_n(M)\coloneq \big\{(p_1,\ldots,p_n)\in M^n\,\big|\, p_i\neq p_j\big\}\]
Understanding the cohomology of configuration spaces, and in particular its behavior as ${n\to \infty}$, is a fundamental problem in topology.
Since cohomology is contravariantly functorial, the cohomology groups $H^m(\Conf_n(M);R)$ together form an FI-module $H^m(\Conf(M);R)$ over $R$. Our main theorem on the cohomology of configuration spaces states that this FI-module is finitely generated. 

\begin{maintheorem}
\label{thm:configurationsfg}
Let $R$ be a Noetherian ring, and let $M$ be a connected orientable manifold of dimension $\geq 2$ with the homotopy type of a finite CW complex (e.g.\ $M$ compact). For any $m\geq 0$, the FI-module $H^m(\Conf(M);R)$ is finitely generated.
\end{maintheorem}

Applying Theorem~\ref{thm:polynomial} and Theorem~\ref{thm:inductive}, respectively, we obtain the following two corollaries.

\begin{theorem}[{\bf Betti numbers of configuration spaces}]
\label{thm:configurationspoly}
Let $k$ be any field, and let $M$ be an connected orientable manifold of dimension $\geq 2$ with the homotopy type of a finite CW complex.   For any $m\geq 0$ there exists a polynomial $P(T)=P_{M,m,k}(T)\in\Q[T]$  so that for all sufficiently large $n$, \[\dim_k H^m(\Conf_n(M);k) = P(n).\]
\end{theorem}

\begin{theorem}[{\bf An inductive description of $H_m(\Conf_n(M);R)$}]
\label{thm:configurationsinductive}
Let $R$ be a Noetherian ring, and let $M$ be a connected orientable manifold of dimension $\geq 2$ with the homotopy type of a finite CW complex. For any $m\geq 0$, there exists $N=N_{M,m}\geq 0$ such that for all $n$:
\[H_m(\Conf_n(M);R)= \colim_{\substack{S\subset [n]\\\abs{S}\leq N}} H_m(\Conf_S(M);R).\]
\end{theorem}

When $k$ has characteristic $0$, Theorem~\ref{thm:configurationspoly} follows from \cite[Theorem~1.9]{CEF}; see Jimenez Rolland \cite[Theorem~1.1]{RJR} for the case $\dim M = 2$.
When $M$ is an open manifold, stronger results hold. In this case Theorem~\ref{thm:configurationspoly} was proved in \cite[Theorem~4.8]{CEF}, in the stronger form that $\dim H^m(\Conf_n(M);k)=P(n)$ for \emph{all} $n\geq 0$. Similarly, when $M$ is open, Theorem~\ref{thm:configurationsinductive} can be deduced from \cite[Theorems~2.24 and 4.7]{CEF}; moreover in this case we can take $N_{M,m}=m$ if $\dim M\geq 3$ \cite[Theorem~4.2]{CEF} and $N_{M,m}=2m$ if $\dim M=2$   \cite[Remark~4.4]{CEF}.

\para{Application 3: Diagonal coinvariant algebras}
Let $k$ be an arbitrary field, let $r$ and $n$ be positive integers, and consider the algebra
\[k[\XX^{(r)}(n)]\coloneq k[x_1^{(1)},\ldots ,x_n^{(1)},\ldots ,x_1^{(r)},\ldots ,x_n^{(r)}]\] of polynomials in $r$ sets of $n$ variables. The permutation group $S_n$ acts on $k[\XX^{(r)}(n)]$ diagonally. Let $I_n$ be the ideal of $k[\XX^{(r)}(n)]$ generated by $S_n$-invariant
 polynomials with vanishing constant term. The \emph{$r$-diagonal coinvariant algebra} is the $k$-algebra 
\[R^{(r)}(n)\coloneq k[\XX^{(r)}(n)]/I_n.\]

The polynomial algebra $k[\XX^{(r)}(n)]$ is naturally endowed with an $r$-fold multi-grading, where a monomial has multi-grading $J=(j_1,\ldots,j_r)$ if its total degree in the variables $x_1^{(i)},\ldots,x_n^{(i)}$ is $j_i$. This multi-grading is $S_n$-invariant, and thus 
descends to an $S_n$-invariant multi-grading
\[R^{(r)}(n)=\bigoplus_JR^{(r)}_J(n)\]
on the $r$-diagonal coinvariant algebra $R^{(r)}(n)$.

When $k$ has characteristic $0$, the $S_n$-representations $R^{(r)}_J(n)$  have been intensively studied. However, when $r>1$ very little is \emph{explicitly}  known about the representations $R^{(r)}_J(n)$, or even their dimensions, except for small $J$ or $n$; see, e.g.\ \cite[\S1]{CEF} for a brief summary.  In \cite[Theorem~1.12]{CEF} it was proved that when $\chr k=0$, the dimension
$\dim_k(R^{(r)}_J(n))$ is a polynomial in $n$ for $n$ sufficiently large.  We are not aware of any 
literature on diagonal coinvariant algebras over fields of positive characteristic.  In this paper, we show that the polynomial behavior of dimension extends to this context.

The key fact which allows us to apply the results of this paper is that for fixed $r$, the coinvariant algebras $R^{(r)}(n)$ can be viewed as forming a co-FI-algebra $R^{(r)}$, as follows. Fix a commutative Noetherian ring $A$  and a positive integer $r$.

If $T$ is a finite set, write $A[\XX^{(r)}(T)]$ for the free commutative $A$-algebra on generators indexed by $[r] \times T$. This algebra naturally has a $\Z_{\geq 0}^r$-valued grading, where the $i$th component records the total degree in the generators $x_{(i,t)}$.
An injection $f\colon S\inj T$ induces a ring homomorphism $f^*\colon A[\XX^{(r)}(T)] \ra A[\XX^{(r)}(S)]$ defined by:
\begin{equation}
\label{eq:polycoFI}
f^*(x_{(i,t)})= \begin{cases}x_{(i,s)}&\text{ if }f(s)=t\\0&\text{ if }t\not\in \im f\end{cases}
\end{equation}
 In other words, $A[\XX^{(r)}]$ can be regarded as a $\Z_{\geq 0}^r$-graded co-FI-algebra, i.e.\ a contravariant functor from $\FI$ to graded $A$-algebras.  

Noting that $\End_{\FI^{\op}}(T)$ is the group of permutations $S_T$, we have an action of $S_T$ on the graded algebra $A[\XX^{(r)}(T)]$.  Define $I_T$ to be the ideal of $S_T$-invariant polynomials with zero constant term, and define $R^{(r)}(T)$ to be the quotient of $A[\XX^{(r)}(T)]$ by $I_T$. Since $I_T$ is a homogeneous ideal, the grading on $A[\XX^{(r)}(T)]$ descends to a $\Z_{\geq 0}^r$-valued grading on $R^{(r)}(T)$. The homomorphisms $f^*$ of \eqref{eq:polycoFI} satisfy $f^*(I_T)\subset I_S$, so they descend to ring homomorphisms $f^*\colon R^{(r)}(T)\to R^{(r)}(S)$. We obtain a $\Z_{\geq 0}^r$-graded co-FI-algebra $R^{(r)}$, which sends the finite set $\{1,\ldots,n\}$ to the usual coinvariant algebra $R^{(r)}(n)$.

We denote by $(R^{(r)})^\vee$ the graded dual of $R^{(r)}$; that is, 
for any $J \in \Z_{\geq 0}^r$ and any finite set $T$, take $(R^{(r)}_J)^\vee(T)$ to be the dual $R$-module $R^{(r)}_J(T)^\vee=\Hom(R^{(r)}_J(T),R)$.  Since $R^{(r)}_J$ is a co-FI-module over $A$, $(R^{(r)}_J)^\vee$ is an FI-module over $A$. Our main theorem on diagonal coinvariant algebras is the following.

\begin{maintheorem}
Let $A$ be a commutative Noetherian ring, and fix $r\geq 1$. For any $J\in\Z_{\geq 0}^r$, the FI-module $(R^{(r)}_J)^\vee$ is finitely generated.
\label{thm:coinvariantsfg}
\end{maintheorem}
Applying Theorem~\ref{thm:polynomial} and Theorem~\ref{thm:inductive}, respectively, we obtain the following two corollaries.

\begin{theorem}[{\bf Multi-graded Betti numbers of diagonal coinvariant algebras}]
\label{thm:coinvariantsdim}
Let $k$ be any field.  For each fixed $r\geq 1$ and fixed $r$-multigrading $J$,  there is an integer-valued polynomial $P(T)=P_{r,J,k}(T)\in\Q[T]$  so that for all sufficiently large $n$, the dimension of the $J$-graded piece of the $r$-diagonal coinvariant algebra is given by \[\dim_k R^{(r)}_J(n)= P(n).\]
\end{theorem}

We do not know any of these polynomials explicitly, except in trivial cases, and it would be very interesting to compute them. It would be intriguing to understand their connection to problems in combinatorics, which has been so fruitful in the characteristic $0$ case. 

\begin{theorem}[{\bf An inductive description of $(R^{(r)}_J)^\vee$}]
\label{thm:coinvariantsinductive}
Let $A$ be a commutative Noetherian ring, and fix $r\geq 1$. For each $J\in\Z_{\geq 0}^r$, there exists $N=N_{r,J,A}\geq 0$ such that for all $n$:
\[R^{(r)}_J(n)^\vee= \colim_{\substack{S\subset [n]\\\abs{S}\leq N}}\, R^{(r)}_J(S)^\vee\]
\end{theorem}

\begin{remark}
Finally, we remark that Theorem~\ref{thm:noetherian} has recently been used by the first author and Putman in \cite{CP}. The main result of that paper is that the Johnson filtration of the mapping class group is ``finitely generated'' in a certain sense (more precisely, generated by elements supported on subsurfaces of uniformly bounded genus). Theorem~\ref{thm:noetherian} is a key technical tool in the proof, and without it the result would not be possible. 
\end{remark}

\para{Acknowledgements} We are grateful to Wolfgang L\"{u}ck for helpful conversations regarding this paper and its relation to \cite[\S16]{Lueck}, to Jesper Grodal for suggesting that Theorem~\ref{thm:inductive} could be formulated as in \eqref{eq:colim}, and to Rita Jimenez Rolland for useful comments and corrections. We are very grateful to the anonymous referee for their thorough and careful reading, and for  thoughtful suggestions that greatly improved the organization of the paper.

\section{Noetherian and polynomial properties of FI-modules}
\label{section:noetherian}

To make this portion of the paper self-contained, we will recall all necessary definitions from our earlier paper \cite{CEF}, and all results that we will use in \S\ref{section:noetherian} will be proved.

\subsection{General results on FI-modules}
Fix a commutative ring $R$, and let $\FIMod$ denote the category of FI-modules over $R$, i.e.\ the category of functors $V\colon \FI\to \RMod$. The category $\FIMod$ is an abelian category, with kernels and cokernels computed pointwise; that is, given $F\colon V\to W$, the FI-modules $\ker(F)$ and $\coker(F)$ satisfy $\ker(F)_S=\ker(F\colon V_S\to W_S)$ and $\coker(F)_S=\coker(F\coloneq V_S\to W_S)$.

If $V$ is an FI-module, we a \emph{sub-FI-module} of $V$ is an FI-module $W$ endowed with an injection $W\inj V$. Identifying $W$ with its image, such a sub-FI-module consists of sub-$R$-modules $W_S\subset V_S$ for each finite set $S$, with the property that $f_*\colon V_S\to V_T$ satisfies $f_*(W_S)\subset W_T$ for all $f\in \Hom_{\FI}(S,T)$.

\para{Finite generation} We recall the  characterizations of finitely-generated FI-modules that we will use in this paper.

\begin{definition}[{\bf Finitely generated FI-modules}]
\label{def:fg}
Let $V$ be an FI-module.   As in \cite[Definitions~2.14 and 2.15]{CEF}, we define an FI-module $V$ 
to be {\em finitely generated} (resp.\ {\em generated in degree $\leq d$})  if there exists a finite set $\{v_1,\ldots,v_k\}\subset \coprod_{n\geq 0}V_n$ (resp.\ a set 
$\{v_i|i\in I\}\subset \coprod_{n\leq d}V_n$) contained in no proper sub-FI-module of $V$. We say that $V$ is {\em generated in finite degree} if it is generated in degree $\leq d$ for some finite $d$.
\end{definition}

It is useful in practice to understand finite generation in terms of ``free'' objects. To this end we make the following definition; see \cite[Definition~2.5]{CEF}.

\begin{definition}[{\bf Free FI-modules}]
\label{def:Md}
For any $d\geq 0$, the FI-module $M(d)$ takes a finite set $S$ to the free $R$-module $M(d)_S$ on the set of injections $[d]\inj S$. In other words, $M(d)=R[\Hom_{\FI}([d],{-})]$; by the Yoneda lemma, $M(d)$ is uniquely determined by the natural identification
\[\Hom_{\FIMod}(M(d),V)\cong V_d.\]
An FI-module is \emph{free} if it is isomorphic to a direct sum $\bigoplus_{i\in I} M(d_i)$.
\end{definition}

Given $v\in V_d$, we denote by $F_v\colon M(d)\to V$ the homomorphism corresponding to $v$ under this identification; conversely, given $F\colon M(d)\to V$ we denote by $v_F\in V_d$ the image of $\id\in M(d)_d$ under $F$. By the Yoneda lemma,  the image $\im(F_v)$ is the sub-FI-module of $V$ defined by $\im(F_v)_S=\spn\{f_*(v)|f\in \Hom_{\FI}([d],S)\}$; this is the smallest sub-FI-module $W\subset V$ for which $v\in W_d$.

\begin{proposition}[{\bf Characterization of finite generation}]
\label{proposition:fg}
Let $V$ be an FI-module.  Then 
\begin{enumerate}
\item $V$ is finitely generated if and only if there exists a surjection
\[\bigoplus_{i=1}^k M(d_i)\surj V\] for some integers $d_i\geq 0$.
\item $V$ is generated in degree $\leq d$ if and only if there exists a surjection
\[\bigoplus_{i\in I} M(d_i)\surj V\qquad\text{ with all }d_i\leq d.\] 
\end{enumerate}
\end{proposition}
The direct sum in the second part of the proposition may be 
infinite, as long as the integers $d_i$ are uniformly bounded by $d$. 
\begin{proof}
The Yoneda lemma guarantees that $M(d)$ is finitely generated by the element $\id\in M(d)_d$. Therefore the free FI-module $\bigoplus_{i\in I} M(d_i)$ is finitely generated if $I$ is finite, and is generated in degree $\leq d$ if $d_i\leq d$ for all $i\in I$.  Conversely, a subset $\{v_i|i\in I\}\subset \coprod_{n\geq 0} V_n$ detemines a canonical map $F=\bigoplus F_{v_i}\colon \bigoplus_{i\in I} M(d_i)\to V$.  The image of $F$ is the smallest sub-FI-module of $V$ containing $\{v_i|i\in I\}$.  The proposition follows.
\end{proof}

In particular, Proposition~\ref{proposition:fg} implies that the quotient of a finitely-generated FI-module is finitely generated, and similarly for generation in degree $\leq d$.

\begin{definition}[{\bf The functor $H_0$} {\cite[Definition~2.18]{CEF}}]
\label{def:H0}
Given an FI-module $V$, the FI-module $H_0(V)$ is the quotient of $V$ defined by:
\[H_0(V)_S\coloneq V_S\ /\ \big\langle \im(f_*\colon V_T\to V_S) \,\big|\  f\colon T\inj S, \,\abs{T}<\abs{S}\big\rangle\]
In other words, the FI-module $H_0(V)$ is the largest quotient of $V$ with the property that for all  $f\colon T\inj S$ with $\abs{T}<\abs{S}$, the map $f_*\colon H_0(V)_T\to H_0(V)_S$ is the zero map. 
\end{definition}

\begin{lemma}
\label{lem:H0reflects}
If $H_0(V)=0$ then $V=0$. Furthermore, the functor $H_0\colon \FIMod\to \FIMod$ reflects surjections: a homomorphism $F\colon V\to W$ is a surjection if and only if $H_0(F)\colon H_0(V)\to H_0(W)$ is a surjection.
\end{lemma}
\begin{proof}
To prove the first claim, assume that $V\neq 0$, and let $n=\inf\{n\in \N|V_n\neq 0\}$. Since $V_T=0$ for any $T$ with $\abs{T}<n$, the quotient defining $H_0(V)_n$ is the quotient by the zero submodule, and thus $H_0(V)_n=V_n\neq 0$.

For the second claim, if $F\colon V\to W$ is surjective, the canonical surjections $V\surj H_0(V)$ and $W\surj H_0(W)$ show that $H_0(F)$ is surjective. In other words, $H_0$ is right-exact. For the converse, right-exactness implies that $\coker H_0(F)=H_0(\coker F)$. If $H_0(F)$ is a surjection we thus have $H_0(\coker F)=\coker H_0(F)=0$. Applying the first claim, we conclude that $\coker F=0$, as desired.
\end{proof}

\begin{lemma}\ 
\label{lem:cokerfg}
Let $V$ be an FI-module. 
\begin{enumerate}[1.]
\item In each row below, the conditions (a), (b), and (c) are equivalent.

\noindent \begin{tabular}{lll}
(a) $V$ is f.g.\ FI-module&(b) $H_0(V)$ is f.g.\ FI-module&(c) $\bigoplus_{n=0}^\infty H_0(V)_n$ is f.g.\ $R$-module\\
(a) $V$ gen.\ in deg.\ $\leq d$&(b) $H_0(V)$ gen.\ in deg.\ $\leq d$&(c) $H_0(V)_n=0$ for all $n>d$\\
(a) $V$ gen.\ in finite deg.&(b) $H_0(V)$ gen.\ in finite deg.&(c) $H_0(V)_n=0$ for $n\gg 0$
\end{tabular}
\item Assume that $V_n$ is a finitely-generated $R$-module for all $n\geq 0$. Then $V$ is finitely generated if and only if $V$ is generated in finite degree.
\end{enumerate}
\end{lemma}
\begin{proof}
\textbf{Part 1.} To start, we observe that each condition in the third row simply asserts that the corresponding condition in the second row holds for some $d\in \N$. Therefore the equivalence of the third row follows from the equivalence of the second row.

 \textbf{(a) $\implies$ (b):} If $V$ is finitely generated or generated in degree $\leq d$, the same is true of any quotient of $V$ by definition. Since $H_0(V)$ is a quotient of $V$, (a) implies (b).

 \textbf{(b) $\implies$ (c):} Let $M=\bigoplus_{i\in I} M(d_i)$ be a free module, and consider a surjection $M\surj H_0(V)$. By the defining property of $H_0(M)$, this map factors through $M\surj H_0(M)\surj H_0(V)$. We have $H_0(M)=\bigoplus_{i\in I} H_0(M(d_i))$. Each summand $H_0(M(d))$ has the property that $H_0(M(d))_T$ vanishes unless $\abs{T}=d$, in which case $H_0(M(d))_T=M(d)_T$ is a free $R$-module of rank $d!$.
  
If $H_0(V)$ is finitely generated, we can take $I$ to be finite, so $\bigoplus_{n=0}^\infty H_0(M)_n$ is a free $R$-module of rank $\sum_{i\in I}{d_i}!$. In particular this $R$-module is finitely generated, and so the same is true of its quotient $\bigoplus_{n=0}^\infty H_0(V)_n$. This shows that (b) implies (c) in the first row. If $H_0(V)$ is generated in degree $\leq d$, we can assume that $d_i\leq d$ for all $i\in I$. In this case $H_0(M)_n=0$ for $n>d$, and so the same is true of its quotient $H_0(V)_n$. This shows that (b) implies (c) in the second row.

 \textbf{(c) $\implies$ (a):} Let $w_i\in \coprod_n H_0(V)_n$ be generators for $\bigoplus_{n=0}^\infty H_0(V)_n$ indexed by $i\in I$, and define $d_i\in \N$ so that $w_i\in H_0(V)_{d_i}$. Set $M\coloneq \bigoplus_{i\in I}M(d_i)$, and define the homomorphism $\pi\colon \bigoplus_{i\in I}M(d_i)\to V$ by sending $\id_{[d_i]}\in M(d_i)_{d_i}$ to any element of $V_{d_i}$ lifting $w_i$. By construction, $H_0(\pi)$ sends $\id_{[d_i]}\in H_0(M(d_i))_{d_i}$ to $w_i\in H_0(V)_{d_i}$. Since $H_0(V)_d$ is generated by the elements $w_i$ for which $d_i=d$, we see that $H_0(\pi)_d\colon H_0(M)_d\to H_0(V)_d$ is surjective for all $d$, i.e.\  $H_0(\pi)$ is surjective. By Lemma~\ref{lem:H0reflects}, the homomorphism $\pi\colon M=\bigoplus_{i\in I}M(d_i)\to V$ is surjective itself.
 
 If $\bigoplus_{n=0}^\infty H_0(V)$ is finitely generated, we can assume that $I$ is finite; in this case, the surjection $\pi\colon M\surj V$ verifies that $V$ is finitely generated. Similarly if $H_0(V)_n=0$ for $n>d$, we can assume that $d_i\leq d$ for all $i\in I$, so $\pi\colon M\surj V$ demonstrates that $V$ is generated in degree $\leq d$. Therefore (c) implies (a).

\textbf{Part 2.}
If $V$ is finitely generated, it is automatically generated in finite degree. Conversely, assume that $V_n$ is a finitely-generated $R$-module for all $n\geq 0$. Since $H_0(V)_n$ is a quotient of $V_n$, it is also a finitely-generated $R$-module. If $V$ is generated in finite degree, then $H_0(V)\sim 0$ by the equivalence of the third row, i.e.\ for some $d\geq 0$ we have $H_0(V)_n=0$ for $n>d$. Therefore the sum $\bigoplus_{n=0}^\infty H_0(V)_n=\bigoplus_{n=0}^d H_0(V)_d$ is finite. It follows that $\bigoplus_{n=0}^\infty H_0(V)_n$ is finitely generated as an $R$-module, so by the equivalence of the first row, $V$ is finitely generated.
\end{proof}

\para{Positive shifts} The following ``shift functors'' on FI-modules will be essential in the proofs of Theorems~\ref{thm:noetherian}, \ref{thm:polynomial}, and \ref{thm:inductive}. Unlike some other parts of the FI-module formalism, these functors would \emph{not} exist if $\FI$ were replaced with an arbitrary diagram category; they depend on the symmetric monoidal structure that comes from taking the disjoint union of finite sets.

We recall from \cite[Definition~2.30]{CEF} the definition of the ``positive shift'' functors \[S_{+a}\colon \FIMod\to\FIMod.\] Let the functor $\disjoint\colon \Sets\times \Sets\to \Sets$ be the coproduct in the category of sets, i.e.\ the disjoint union of sets. This should be formalized in some fixed functorial way; for example, we could take $S\disjoint T\coloneq (\{0\}\times S)\cup (\{1\}\times T)$. But since the coproduct is unique up to canonical isomorphism, nothing will depend on the details of this definition.

Since $f\disjoint g\colon S\disjoint S'\to T\disjoint T'$ is injective if $f\colon S\to T$ and $g\colon S'\to T'$ are injective, $\disjoint$ restricts to a functor $\disjoint\colon \FI\times \FI\to \FI$. Since $S$ and $T$ are canonically identified with subsets of $S\disjoint T$, we will often abuse notation and treat $S$ and $T$ as subsets of $S\disjoint T$.

\begin{definition} For $a\geq 0$, let $[-a]$ denote the set $\{-1,\ldots,-a\}$, and let $\Xia\colon \FI\to\FI$ be the functor \[\Xia\colon \FI\to \FI\qquad\qquad \Xia\coloneq{-}\disjoint [-a].\] Explicitly, $\Xia S$ is the finite set $S\disjoint [-a]$, and $\Xia f\colon \Xia S\inj \Xia T$ is $f\disjoint \id_{[-a]}$, the extension of $f$ by the identity on $[-a]$. Let $i_{-a}\colon [-a]\inj [-(a+1)]$ be the standard inclusion $i_{-a}\colon \{-1,\ldots,-a\}\inj \{-1,\ldots,-a,-(a+1)\}$.
\end{definition} Our choice of the set $\{-1,\ldots,-a\}$ for $[-a]$ is irrelevant, since the disjoint union $S\disjoint T$ is defined even if $S$ and $T$ are not disjoint; it is chosen just for psychological purposes, to minimize  collision with the sets the reader likely has in mind. Any other set of cardinality $a$ would work equally well.

\begin{definition}[{\bf Positive shift functor $S_{+a}$}]
\label{def:shiftpos}
Given an FI-module $V$ and an integer $a\geq 1$, the functor $S_{+a}\colon \FIMod\to \FIMod$ is defined by $S_{+a}={-}\circ \Xia$; that is, the FI-module $S_{+a}V$ is the composition \[S_{+a}V\coloneq V \circ \Xia\colon \FI\overset{\Xia}{\longrightarrow} \FI\overset{V}{\longrightarrow} \RMod.\]
Since kernels and cokernels are computed pointwise, $S_{+a}$ is an exact functor.
\end{definition}

\begin{remark}Comparing the $S_n$-representation $(S_{+a}V)_n$ with the $S_{n+a}$-representation $V_{n+a}$, we have an isomorphism of $S_n$-representations
\[(S_{+a}V)_n\cong \Res^{S_{n+a}}_{S_n} V_{n+a}.\]  Indeed, the effect of the functor $S_{+a}$ is to perform this restriction consistently for all $n$, in such a way that the resulting representations still form an FI-module.
\end{remark}

\begin{definition}[{\bf The morphism $X_a\colon V\to S_{+a}V$}]
\label{def:Xa}
The natural inclusion $\iota_T$ of $T$ into $\Xia T = T\disjoint [-a]$ induces a natural transformation $\id_{\FI}\implies \Xia$. For any FI-module $V$, this yields a natural homomorphism of FI-modules
\begin{equation}
\label{eq:SVV}
\X_a\colon V\to S_{+a}V.
\end{equation}
Explicitly, $\X_a$ is defined by
\[\X_a\colon V_T\overset{V(\iota_T)}{\longrightarrow}V_{T\disjoint[-a]}=(S_{+a}V)_T.\]
Similarly, the inclusion $\id\disjoint i_{-a}\colon T\disjoint [-a]\inj T\disjoint [-(a+1)]$ induces a natural homomorphism $\Y_a\colon S_{+a}V \to S_{+(a+1)}V$ satisfying $\X_{a+1}=\Y_a\circ \X_a\colon V\to S_{+(a+1)}V$.
\end{definition}

\begin{definition}[{\boldmath$V\approxeq W$}]  If $V$ and $W$ are FI-modules, we write $V\approxeq W$ if $S_{+a} V \cong S_{+a} W$ for some $a \geq 0$.  This notation is most often used in this paper in the form $V \approxeq 0$, which simply means that $V_n$ vanishes for all sufficiently large $n$. For example, in this language Lemma~\ref{lem:cokerfg}.1 states that $V$ is generated in finite degree if and only if $H_0(V)\approxeq 0$. If $X$ and $Y$ are FI-simplicial complexes, we write $X\approxeq Y$ if for each $k$, the $k$-skeleta satisfy $S_{+a} X^{(k)}\cong S_{+a} Y^{(k)}$ for some $a$ depending on $k$.
\end{definition}

\subsection{The Noetherian property (Proof of Theorem~\ref{thm:noetherian})}
The shift functors $S_{+a}$ take on a particularly simple form when applied to the FI-modules $M(d)$. 
\begin{proposition}
\label{pr:SaMd}
For any $a\geq 0$ and any $d\geq 0$, there is a natural decomposition
\begin{equation}
\label{eq:Mderivative}
S_{+a}M(d)=M(d)\oplus Q_a
\end{equation}
where $Q_a$ is a free FI-module finitely generated in degree $\leq d-1$.
\end{proposition} Although this proposition appears unassuming, this is the key combinatorial fact about the category $\FI$ that makes possible our approach to the Noetherian property for $\FI$-modules. For comparison, if $\FI$ were replaced by the category of finite-dimensional $\FF$--vector spaces and linear injections for some field $\FF$, the analogous proposition would not hold (even if the field $\FF$ were finite), and so our proof of the Noetherian property does not extend to this case.
\begin{proof}[Proof of Proposition~\ref{pr:SaMd}]
Recall that a basis for $M(d)_S$ is given by $\Hom_{\FI}([d],S)$, so a basis for $(S_{+a}M(d))_S$ consists of the injections $f\colon [d]\hookrightarrow S\disjoint [-a]$.
We can stratify these according to the subset $T=f^{-1}([-a])\subset [d]$ and the restriction $f|_T\colon T\hookrightarrow [-a]$. Given $g\colon S\inj S'$, the  map $g_*\colon S_{+a}M(d)_S\to S_{+a}M(d)_{S'}$ is induced by the composition \[g_*f=(g\disjoint \id_{[-a]})\circ f,\] so the subset $f^{-1}([-a])=T$ and the restriction $f|_T$ are not changed by the composition $f\mapsto g_*f$. Therefore our stratification of $S_{+a}M(d)_S$ in fact defines a  decomposition of $S_{+a}M(d)$ as a direct sum of  FI-modules.

For fixed $T\subset [d]$ and $h\colon T\hookrightarrow [-a]$, let $M^{T,h}\subset S_{+a} M(d)$ be spanned by those $f\colon [d]\inj S\disjoint [-a]$ satisfying $f^{-1}([-a])=T$ and $f|_T=h$.  These injections $f$ are distinguished by the restrictions $f|_{[d]-T}\colon [d]-T\inj S$, and we have $(g_*f)|_{[d]-T}=g\circ f|_{[d]-T}$. Choosing a bijection $[d]-T\cong [d-\abs{T}]$, we obtain an isomorphism $M^{T,h}\cong M(d-\abs{T})$, and thus a decomposition
\[S_{+a}M(d)=\bigoplus_{T\subset [d]} M(d-\abs{T})\tensor_R R[\Hom_{\FI}(T,[-a])]\] which is natural up to the choice of  identifications $[d]-T\cong [d-\abs{T}]$. In particular, the summand with $T=\emptyset$ is canonically isomorphic to $M(d)$; singling out this factor gives the claimed decomposition.
\end{proof}

\begin{corollary}[{\cite[Proposition~2.31]{CEF}}]
\label{cor:shiftgeneration}
If $V$ is generated in degree $\leq d$, then $S_{+a}V$ is generated in degree $\leq d$. Conversely, if $S_{+a}V$ is generated in degree $\leq d$, then $V$ is generated in degree $\leq d+a$.
\end{corollary}
\begin{proof}
Choosing a surjection $\bigoplus M(d_i)\surj V$ with $d_i\leq d$, and given that $S_{+a}V$ is exact, it suffices to prove the first claim for $V=M(d_i)$. This follows immediately from Proposition~\ref{pr:SaMd}.

For the second claim we use Lemma~\ref{lem:cokerfg}.1, which says that $S_{+a}V$ is generated in degree $\leq d$ if and only if $H_0(S_{+a}V)_n=0$ whenever $n>d$. We will exhibit in the next paragraph a surjection of $R$-modules $H_0(S_{+a}V)_n\surj H_0(V)_{n+a}$. From this surjection we deduce that $H_0(V)_m=0$ for all $m>d+a$; applying Lemma~\ref{lem:cokerfg}.1 again, we conclude that $V$ is generated in degree $\leq d+a$ as desired.

We now exhibit the claimed surjection, in the form $H_0(S_{+a}V)_T\surj H_0(V)_{T\disjoint [-a]}$. By Definition~\ref{def:H0}, $H_0(S_{+a}V)_T$ is the quotient of $(S_{+a}V)_T=V_{T\disjoint [-a]}$ by \[
\big\langle \im (f\sqcup \id_{[-a]})_*\colon V _{S\disjoint [-a]}\to V_{T\disjoint [-a]}\,\ \big|\,\qquad f\colon S\inj T, \abs{S}<\abs{T}\big\rangle\] while $H_0(V)_{T\disjoint [-a]}$ is the quotient of $V_{T\disjoint [-a]}$ by \[
\big\langle \im g_*\colon V_{S'}\to V_{T\disjoint [-a]}\qquad\quad\,\big|\,\  g\colon S'\inj T\disjoint [-a], \abs{S'}<\abs{T}+a\big\rangle.\]
The former is contained in the latter, so $H_0(V)_{T\disjoint [-a]}$ is a quotient of $H_0(S_{+a}V)_T$ as claimed.
\end{proof}

\begin{definition}
\label{def:pia}
We define
\begin{equation}
\label{eq:shiftprojection}
\pi_a\colon S_{+a}M(d)\twoheadrightarrow M(d)
\end{equation} 
 to be the projection determined by \eqref{eq:Mderivative}. Concretely, a basis for $(S_{+a}M(d))_T$ consists of injections $[d]\hookrightarrow T\disjoint[-a]$, and the projection 
$\pi_a$ simply sends to 0 any injection whose image is not contained in $T$.
\end{definition}

The projection $\pi_a$ is related to the decomposition of $M(d)_n$ given by splitting up the injections $\{1,\ldots,d\}\hookrightarrow\{1,\ldots,n\}$ according to their image. Each $d$-element subset $S$ of $\{1,\ldots,n\}$ gives a summand isomorphic to $M(d)_d$, yielding the decomposition as $R$-modules
\begin{equation}
\label{eq:Mdsplitting}
M(d)_n\simeq M(d)_d^{\oplus \binom{n}{d}}.
\end{equation}
In degree $d$, the projection $\pi_a$ yields a map from $(S_{+(n-d)}M(d))_d\simeq M(d)_n$ to $M(d)_d$; this is just the projection of \eqref{eq:Mdsplitting} onto a single factor of the right side.

\begin{proof}[Proof of Theorem~\ref{thm:noetherian}]
We prove by induction on $d\in \N$ that if $V$ is a  FI-module finitely generated in degree $\leq d$, then every sub-FI-module of $V$ is finitely generated. Such an FI-module $V$ is a quotient of a finite direct sum of FI-modules $\bigoplus_{i=1}^k M(d_i)$ with $d_i\leq d$. Since the Noetherian property descends to quotients and is preserved under direct sum, it suffices to prove the theorem for $V=M(d_i)$, and by induction it suffices to prove it for $V=M(d)$.

\para{Reduction to $W^a$}
Fix a sub-FI-module  $W$ of $M(d)$; our goal is to prove that $W$ is finitely generated. For each $n\in \N$ we have that $M(d)_n$ is a finitely-generated $R$-module; since $R$ is a Noetherian ring, its submodule $W_n$ is also finitely generated as an $R$-module. Therefore by Lemma~\ref{lem:cokerfg}.2 it suffices to prove that $W$ is generated in finite degree. By Corollary~\ref{cor:shiftgeneration} it suffices to prove that $S_{+a}W$ is finitely generated for some $a\geq 0$.

Let us therefore consider the FI-module $S_{+a}W$.
For any $a\geq 0$, the decomposition from Proposition~\ref{pr:SaMd} gives an exact sequence
\[0\to Q_a\to S_{+a}M(d)\to M(d)\to 0.\] Since $S_{+a}$ is exact, we can think of $S_{+a}W$ as a sub-FI-module of $S_{+a}M(d)$. Thus the above induces an exact sequence
\[0\to W_{Q,a}\to S_{+a}W\to W^a\to 0\]
where $W_{Q,a}\coloneq Q_a\cap(S_{+a}W)$ and $W^a\coloneq \pi_a(S_{+a}W)\subset M(d)$.

For any $a$ we know that $W_{Q,a}$ is a sub-FI-module of $Q_a$, which is finitely generated in degree $\leq d-1$ by Proposition~\ref{pr:SaMd}. Therefore we can apply the inductive hypothesis to conclude that $W_{Q,a}$ is finitely generated for any $a$. To prove that $S_{+a}W$ is finitely generated, it thus suffices to show that $W^a$ is finitely generated. We will do this, and thus prove the theorem, by showing that there exists some $N\geq 0$ such that $W^N$ is finitely generated in degree $\leq d$.

\para{Finding $N$ such that $W^N$ is generated in degree $\leq d$} The sequence of sub-FI-modules $W^a\subset M(d)$ is increasing: $W^a\subset W^{a+1}$. Indeed, the map $\Y_a\colon S_{+a} M(d) \ra S_{+(a+1)} M(d)$ of Definition~\ref{def:Xa} satisfies $\pi_{a+1}\circ \Y_a=\pi_a$ and $\Y_a(S_{+a} W) \subset S_{+(a+1)}(W)$, from which it follows that $W^a \subset W^{a+1}$.  Let $W^\infty$ denote the sub-FI-module $\bigcup_a W^a \subset M(d)$.

We will show below that $W^\infty$ is  generated by $W^\infty_d$ (that is, the only submodule $X\subset W^\infty$ with $X_d=W^\infty_d$ is $X=W^\infty$). Since $W^\infty_d$ is a sub-$R$-module of $M(d)_d \cong R[S_d]$, it is itself finitely generated as an $R$-module, so the claim implies that $W^\infty$ is finitely generated in degree $\leq d$.
Moreover, the chain
\beq
W_d = W^0_d \subset W^1_d \subset W^2_d \subset \ldots \subset W^\infty_d=\bigcup_a W^a_d
\eeq
is a chain of $R[S_d]$-submodules of $M(d)_d \cong R[S_d]$.  Since $R[S_d]$ is a finitely-generated $R$-module and $R$ itself is Noetherian, there must be some $N$ such that $W^N_d = W^\infty_d$. Since $W^\infty$ is generated by $W^\infty_d$, it follows that $W^N=W^\infty$, and thus that $W^N$ is finitely generated in degree $\leq d$ as desired. 

\para{Proving that $W^\infty$ is generated by $W^\infty_d$} Let us investigate the sub-FI-modules $W^a\subset M(d)$. Expanding the definition of $W^a$, we have the following concrete condition: an element
\beq
x = \sum_{f\colon [d] \inj T} r_f f \in M(d)_T
\eeq
lies in $W^a$ if and only if there is an element
\begin{equation} \label{eq:witness}
\phantom{\subset M(d)_{T\disjoint [-a]}}w = \sum_{g\colon [d] \inj T\disjoint [-a]} r'_g g \in W_{T\disjoint [-a]}\subset M(d)_{T\disjoint [-a]}
\end{equation}
such that $r'_g = r_g$ whenever the image of $g$ lies in $T$. The element $x\in M(d)_T$ lies in $W^\infty$ if there is \emph{some} $a\geq 0$ and some $w\in W_{T\disjoint [-a]}$ for which \eqref{eq:witness} holds.

For each $a\geq 0$, let $U^{a}$ be the smallest sub-FI-module of $W^a$ containing $W^a_d$. We will show that for any $a\geq 0$ and any $n\leq a+d$ we have \[W^{a+d-n}_n\subset U^{a}_n\subset M(d)_n.\]
Given $x \in W^{a+d-n}_n\subset M(d)_n$, write $x=\sum_{f\colon [d]\inj [n]} r_f f$ as above, and for each  subset $S\subset [n]$ of cardinality $d$, denote by $x_S$ the sum
\beq
x_S\coloneq \sum_{\im f = S} r_f f\in M(d)_S.
\eeq We have $x=\sum_S i_S(x_S)$, where $i_S\colon S\inj [n]$ is the natural inclusion. 

Since $x\in W^{a+d-n}_n$, there exists $w\in W_{[n]\disjoint [-(a+d-n)]}$ so that, writing \[w=\sum_{g\colon [d]\inj[n]\disjoint[-(a+d-n)]}r'_g g\] as in \eqref{eq:witness}, we have $r'_g=r_g$ for all $g$ with image contained in $[n]$. But then it is \emph{a fortiori} true that $r'_g=r_g$ for all $g$ with $\im g=S$. Choosing a bijection between $([n]-S)\disjoint [-(a+d-n)]$ and $[-a]$, we can think of $w$ as an element of $W_{S\disjoint [-a]}$ which witnesses  that $x_S\in W^a_S$ as in  \eqref{eq:witness}.

Since $\abs{S}=d$, we have $W^a_S=U^a_S$ by definition. Since $x=\sum_S i_S(x_S)$, we conclude that $x\in U^a$. Since this holds for all $x\in W^{a+d-n}_n$, we see that $W^{a+d-n}_n$ is contained in $U^a$, as claimed above.  Passing to the limit as $a\to \infty$ and setting $U^\infty\coloneq \bigcup_a U^a$, we see that $W^\infty_n$ is contained in $U^\infty_n$ for all $n\in \N$. Since $U^\infty$ is contained in $W^\infty$ by definition, this shows that $U^\infty=W^\infty$; in other words, $W^\infty$ is generated by $W^\infty_d$, which was the remaining claim to be proved.
\end{proof}

\subsection{Dimensions of f.g.\ FI-modules (Proof of Theorem~\ref{thm:polynomial})}
Let $V$ be an FI-module.  The {\em torsion submodule} of $V$, denoted $T(V)$, consists of those $v\in V_S$ for which $f_*(v)=0$ for some finite set $S'$ and some (whence every) $f\in \Hom_{\FI}(S,S')$.  Alternatively, it can be written as
\beq
T(V) = \bigcup_{a \geq 0} \ker( \X_a\colon V \ra S_{+a} V).
\eeq
We say that $V$ is {\em torsion free} if $T(V) = 0$.  It is clear that $V/T(V)$ is always torsion free.  When $k = \C$, the functor $T$ is discussed by Sam and Snowden~\cite[\S 4.4]{SS}, where it appears as the left exact functor $H^0_{\mathfrak{m}}$ whose derived functors provide a local cohomology theory for FI-modules.

\begin{lemma}
\label{lem:torsion}
If $V$ is a finitely-generated FI-module over a Noetherian ring, then $T(V) \approxeq 0$; in other words, $T(V)_n=0$ for all sufficiently large $n$.
\end{lemma}

\begin{proof}
By Theorem~\ref{thm:noetherian}, the submodule $T(V)$ is finitely generated, say by $v_1,\ldots,v_k$ with $v_i\in V_{d_i}$. Therefore $T(V)_S$ is spanned by $\bigcup_i\{f_*(v_i)|f\colon [d_i]\to S\}$ for every finite set $S$.
For each $i$ there exists some $a_i$ for which $v_i\in \ker(X_{a_i}\colon V\to S_{+a_i}V)$. Setting  $M_i=d_i+a_i$, this implies that $f_*(v_i) = 0$ for any $f\in \Hom([n_i],S)$ with $\abs{S} \geq M_i$.
Taking $M\coloneq \max M_i$, we see that as long as $\abs{S}\geq M$ we have $f_*(v_i)=0$  for any $i$ and any $f\in \Hom([n_i],S)$. Since these elements generate $T(V)_S$, this implies that $T(V)_S=0$ whenever $\abs{S}\geq M$. Therefore $T(V)\approxeq 0$ as desired.\end{proof}

\begin{proof}[Proof of Theorem~\ref{thm:polynomial}]
We will prove the stronger statement that if $V$ is finitely generated in degree $\leq d$, then $\dim_k V_n$ is eventually equal to an integer-valued polynomial of degree $\leq d$. The proof is by induction on $d$. We say that $V$ is generated in degree $\leq -1$ if $V=0$, and that a polynomial is of degree $\leq -1$ if it vanishes; we can thus take as our base case $d=-1$.

It follows from Lemma~\ref{lem:torsion} that the torsion-free quotient $V' = V/T(V)$ has $\dim_k V'_n = \dim_k V_n$ for $n \gg 0$, and being a quotient of $V$ we know that $V'$ is still generated in degree $\leq d$.  Thus we can assume without loss of generality that $V$ is torsion-free.  Under this assumption, the natural map $\X_1\colon V\to S_{+1}V$ is injective.  Let $DV$ denote the cokernel of this map.

We show that $DV$ is finitely generated in degree $\leq d-1$.
First, if $V=M(m)$ for some $m\leq d$, then Proposition~\ref{pr:SaMd} shows that $DV=Q_1$ is finitely generated in degree $\leq m-1$.  A general FI-module $V$ is finitely generated in degree $\leq d$ if there is a surjection $M\coloneq \bigoplus_{i=1}^k M(d_i)\surj V$ with $d_i\leq d$. Since $S_{+1}$ is exact, $S_{+1}M\surj S_{+1}V$ is surjective. We conclude that the quotient $DM$ surjects to $DV$, so the general claim follows from the special case $V=M(m)$ proved above.

By induction, we can conclude that $\dim_k DV_n$ is eventually a polynomial of degree at most $d-1$.
But if we write $f(n)$ for $\dim_k V_n$, we have
\beq
\dim DV_n =  \dim (S_{+1}V)_n-\dim V_n = f(n+1) - f(n)
\eeq
Therefore we have just proved that the discrete derivative $f(n+1) - f(n)$ is eventually a polynomial of degree at most $d-1$. It follows that $f(n)$ is eventually a polynomial of degree at most $d$, which is the statement to be proved.
\end{proof}

\subsection{Inductive description of f.g.\ FI-modules (Proof of Theorem~\ref{thm:inductive})}
Our goal in this section is to understand when an FI-module $V$ admits an inductive description
\[V_n=\colim_{S\subsetneq [n]} V_S,\] at least for large enough $n$. As we will see, such a description is equivalent to the exactness of the sequence
\[\tSminus{2}V\to \tSminus{1}V\to V\to 0,\] where $\tSminus{2}$ and $\tSminus{1}$ are certain functors defined below. In fact, we will define an entire complex 
\begin{equation*}
\tSminus{\ast}V=\quad\cdots\to\tSminus{a}V\to\tSminus{(a-1)}\to \cdots\to \tSminus{2}V\to \tSminus{1}V\to V\to 0.
\end{equation*}
This complex also appeared in \cite[\S4]{P}, where it arose in a very different way, and will be used in Section~\ref{sec:congruence}.

\para{Ordered negative shifts $B_a$}
We begin by defining functors $B_a\colon \FIMod\to \FIMod$, and a complex \begin{equation*}
B_{*}V=\quad\cdots\to B_aV\to B_{a-1}V\to \cdots\to B_2V\to B_1V\to V\to 0
\end{equation*}
We will then define $\tSminus{a}V$ as a quotient of $B_aV$, in such a way that the complex $B_{*}V$ descends to the desired complex $\tSminus{\ast}V$.

\begin{definition}[\textbf{Ordered negative shift functor $B_a$}]
\label{def:ordnegshift}
Given an FI-module $V$ and an integer $a\geq 0$, we define $B_a V$ to be the FI-module which maps a set $S$ to the direct sum \begin{equation}
\label{eq:tempnegshift}
(B_a V)_S = \bigoplus_{f\colon [a]\hookrightarrow S}V_{S-f([a])}.
\end{equation}
We denote by $(B_a V)_{S,f}$ the summand corresponding to $f$ in the decomposition \eqref{eq:tempnegshift}. 
The map $g_*\colon (B_aV)_S\to (B_aV)_T$ induced by $g\colon S\inj T$ takes the factor $(B_aV)_{S,f}$ to the factor $(B_aV)_{T,g\circ f}$ by the map $(g|_{S-f([a])})_*\colon V_{S-f([a])}\to V_{T-g\circ f([a])}$. The assignment $V\mapsto B_aV$ defines the exact functor $B_a\colon \FIMod\to \FIMod$.
\end{definition}

\begin{lemma}
\label{lem:BaMd}
For any $d\geq0$ and any $a\geq 0$ there is a natural isomorphism $B_aM(d)\simeq M(a+d)$.
\end{lemma}
\begin{proof}
Given an injection $f\colon[a]\inj S$, a basis for the summand $(B_aM(d))_{S,f}=M(d)_{S-f([a])}$ consists of the injections $f'\colon [d]\inj S-f([a])$. Therefore a basis for $(B_aM(d))_S$ consists of pairs $(f\colon [a]\inj S, f'\colon [d]\inj S)$ with $f([a])\cap f'([d])=\emptyset$. Fixing an isomorphism $[a]\disjoint [d]\simeq [a+d]$, the isomorphism $B_aM(d)\to M(a+d)$ is defined by sending $(f,f')\in B_aM(d)_S$ to $f\disjoint f'\colon [a+d]\inj S$. An injection $g\colon S\inj T$ acts on $B_aM(d)$ by $g_*(f,f')=(g\circ f,g\circ f')$, so since $(g\circ f)\disjoint (g\circ f')=g\circ(f\disjoint f')$, we indeed have an isomorphism of FI-modules $B_aM(d)\simeq M(a+d)$.
\end{proof}

Since $B_a$ is exact, it follows from Lemma~\ref{lem:BaMd} that if $V$ is generated in degree $\leq d$, then $B_aV$ is generated in degree $\leq a+d$. In particular, we have the following corollary.

\begin{corollary}
\label{cor:BVfg}
If $V$ is finitely generated, then $B_a V$ is finitely generated for any $a\geq 0$.
\end{corollary}

\para{The complex $B_\ast V$}
We can package all the $B_a V$ together as a single object $B_V$. Given an FI-module $V$, let  $B_V$ be the functor $B_V\colon\FI^{\op}\times \FI\to \RMod$ defined by:
\beq
B_V(U,S) = \bigoplus_{f\colon U \inj S} V_{S-f(U)}
\eeq
Write $B_V(U,S,f)$ for the corresponding factor of $B_V(U,S)$. For a morphism $g\in \Hom_{\FI}(S,T)$, the map $g_*\colon B_V(U,S)\to B_V(U,T)$ is given on factors $B_V(U,S,f)\to B_V(U,T,g\circ f)$ as described in Definition~\ref{def:ordnegshift}. To a morphism $h\in \Hom_{\FI^{\op}}(U,Z)$, i.e.\ an injection $h\colon Z\inj U$, we associate the map $h^*\colon B_V(U,S)\to B_V(Z,S)$ is given on factors $B_V(U,S,f)\to B_V(Z,S,f|_Z)$ by the map map $i_*\colon V_{S-f(U)}\to V_{S-f(Z)}$, where $i$ denotes the inclusion of the subset $S-f(U)$ into $S-f(Z)$.
Then $B_a V$ is the FI-module $B_V([a],{-})$.

For $1\leq i\leq a$, let $s_i$ be the order-preserving inclusion from $[a-1]$ to $[a]$ whose image misses $i$. Considering $s_i$ as a morphism in $\FI^{\op}$ from $[a]$ to $[a-1]$, it naturally induces a map of FI-modules $B_V([a],{-})\to B_V([a-1],{-})$, i.e. a map $d_i\colon B_a V\to B_{a-1} V$. 
  The functors $B_a V$ fit together into a natural complex of FI-modules
\begin{equation}
\label{eq:Bcomplex}
B_{*}V\coloneq\qquad \cdots\to B_{3} V\to B_{2} V\to B_{1}V\to V\to 0
\end{equation}
with differential $d\colon B_{a}V\to B_{(a-1)}V$ given by the alternating sum $\sum (-1)^i d_i$.
  The familiar identity $s_i\circ s_j=s_{j+1}\circ s_i$ of inclusions $[a-2]\inj [a]$ (for $1\leq i\leq j<a$) implies that $d_j\circ d_i=d_i\circ d_{j+1}$, from which it follows that $d^2=0$.

\para{Twisted negative shifts $\tSminus{a}$}
We defined $B_V\colon \FI^{\op}\times \FI\to \RMod$ above, and noted that $B_aV$ is given by $B_V([a],{-})$. Therefore the group of automorphisms  $\Aut_{\FI^{\op}}([a])$  gives a natural action of $S_a$ on $B_aV$ by FI-module automorphisms; explicitly, this action permutes the factors $(B_aV)_{S,f}$ by precomposing the injections $f\colon [a]\inj S$ with permutations of $[a]$. 

\begin{definition}[\textbf{Negative shift functor $\tSminus{a}$}]
\label{def:negshift}
Let $\epsilon_a$ denote the \emph{sign representation} of $S_a$; that is, the $R[S_a]$-module which as an $R$-module is simply $R$, and on which a permutation $\sigma$ acts by $(-1)^\sigma$.
We define
\beq
\tSminus{a} V = B_a V \tensor_{R[S_a]} \epsilon_a.
\eeq
\end{definition}

The effect is that $(\tSminus{a}V)_S$ has one summand $V_T$ for each subset $T\subset S$ with $\abs{T}=\abs{S}-a$, on which permutations of $T$ act as they usually do on $V_T$, and on which permutations of $S-T$ act by their sign. The surjection $R[S_a]\twoheadrightarrow \epsilon_a$ induces a surjection $B_a V\twoheadrightarrow \tSminus{a} V$, so as a consequence of Corollary~\ref{cor:BVfg} we obtain the following.
\begin{lemma}
\label{lem:shiftfg}
If $V$ is finitely generated, then $\tSminus{a}V$ is finitely generated for any $a\geq 1$.
\end{lemma}

\para{The complex $\tSminus{*}V$}
Since the differential $d\colon B_aV\to B_{a-1}V$ was defined as the alternating sum $\sum (-1)^i d_i$, it descends to a differential $d\colon\tSminus{a}V\to \tSminus{(a-1)}V$. We thus obtain a natural complex of FI-modules
\begin{equation}
\label{eq:shiftcomplex}
\tSminus{*}V\coloneq\qquad \cdots\to \tSminus{3} V\to \tSminus{2} V\to \tSminus{1}V\to V\to 0.
\end{equation}

\begin{remark} In a sequel to this paper~\cite{CE} we interpret the homology $H_a(\tSminus{*}V)$ of this complex as the ``FI-module homology'' of $V$, and use this to give quantitative bounds on the stable range in Theorems~\ref{thm:congruencepoly} and \ref{thm:congruenceinductive}. We point out that the same complex was considered independently by Putman~\cite{P}, where the degree-wise slices $(\tSminus{*}V)_M$ appear as the ``$M$-central stability chain complex'' \cite[Lemma 4.4]{P}  in the context of central stability for representations of $S_n$ over a field.

In fact, the complex $\tSminus{*}V$ seems to arise naturally from \emph{three} independent perspectives: 1) as the ``central stability chain complex'' which governs the central stability of a sequence of $S_n$-representations \cite{P}; 2) as the Koszul resolution which computes the FI-module homology $H_i^{\FI}$ of an FI-module $V$ \cite{CE}; and 3) from the equivariant chains of the complex of split unimodular sequences constructed by Charney in \cite{Charney}. The approach of Putman~\cite{P} rests on the relation between the first and third, while the approach of \cite{CE} is based on the second and third.

In a sense, the present paper uses all three perspectives on the complex $\tSminus{*}V$: the first in Lemmas~\ref{lem:H0identify} and \ref{lem:H1identify} and the proof of Theorem~\ref{thm:inductive}, the second in Proposition~\ref{prop:homologygraded}, and the third in the proof of Theorem~\ref{thm:congruencefg}. However, to keep this paper self-contained, we will derive all necessary properties of $H_a(\tSminus{*}V)$ from their definition in terms of the complex $\tSminus{*}V$ of \eqref{eq:shiftcomplex}. 
\end{remark}

\para{Identifying $H_0(\tSminus{*}V)$ and $H_1(\tSminus{*}V)$}
For a finite set $T$, let $C^{\subseteq T}$ denote the poset of subsets $S\subset T$ under inclusion. We can consider $C^{\subseteq T}$ as a subcategory of $\FI$, and in fact the inclusions $i^S\colon S\inj T$ let us consider $C^{\subseteq T}$ as a subcategory of the over-category $\FI/T$. Therefore for any FI-module $V$ and any sub-poset $D\subset C^{\subseteq T}$, we have a $D$-indexed diagram $V_S$, and the maps $i^S_*\colon V_S\to V_T$ induce a canonical homomorphism \[\colim_D V_S\to V_T.\]
In general, this homomorphism of $R$-modules will be neither injective or surjective. However, the following lemmas demonstrate that when $D=C^{\subsetneq T}$, the injectivity and surjectivity of the homomorphism
\begin{equation}
\label{eq:colimmap}
\colim_{S\subsetneq T}V_S\to V_T
\end{equation}
are computed by $H_1(\tSminus{*}V)$ and $H_0(\tSminus{*}V)$ respectively.
\begin{lemma}
\label{lem:H0identify}
Let $V$ be an FI-module. Then $H_0(\tSminus{\ast}V)=H_0(V)$; moreover, for each finite set $T$ we have \[H_0(\tSminus{\ast}V)_T=\coker\big(\colim_{S\subsetneq T}V_S\to V_T\big)=H_0(V)_T.\]
\end{lemma}
\begin{proof}
Comparing the definitions of the three $R$-modules in question, we have by definition that $H_0(\tSminus{\ast}V)_T=\coker(\tSminus{1}V\to V)_T$ is the quotient of $V_T$ by the submodule
\[\big\langle \im i^S_*\colon V_S\to V_T\,\big|\, S\subset T,\ \abs{S}=\abs{T}-1\big\rangle,\]
while $\coker\big(\colim_{S\subsetneq T}V_S\to V_T\big)$ is the quotient of $V_T$ by the submodule
\[\big\langle \im i^S_*\colon V_S\to V_T\,\big|\, S\subset T,\ \abs{S}<\abs{T}\big\rangle,\]
and $H_0(V)_T$ is the quotient of $V_T$ by the submodule
\[\big\langle \im f_*\colon V_U\to V_T\,\big|\, f\colon U\inj T,\ \abs{U}< \abs{T}\big\rangle.\]
Therefore it suffices to prove that these three submodules coincide.

By definition, the first submodule is contained in the second, and the second submodule is contained in the third. Conversely, for any  $f\colon U\inj T$ with $\abs{U}<\abs{T}$, set $S=f(U)$. Then $f$ factors as $i^S\circ f'$, where $f'\colon U\to f(U)=S$ is the co-restriction of $f$. It follows that $\im f_*$ is contained in $\im i^S_*$, demonstrating that the third submodule is contained in the second. Similarly, for any $S\subset T$ with $\abs{S}<\abs{T}$, choose $S'$ such that $S\subset S'\subset T$ and $\abs{S'}=\abs{T}-1$. Then $i^S=i^{S'}\circ i^{S,S'}$, demonstrating that the second submodule is contained in the first.  
\end{proof}

\begin{lemma}
\label{lem:H1identify}
Let $V$ be an FI-module. For each finite set $T$,
\[H_1(\tSminus{*}V)_T=\ker\big(\colim_{S\subsetneq T}V_S\to V_T\big).\]
\end{lemma}
\begin{proof}
Let $C^{\subsetneq T}$ be the poset of proper subsets $S\subsetneq T$ under inclusion, and let $D^T$ be the poset of subsets $S\subset T$ with $\abs{T}-2\leq \abs{S}\leq \abs{T}-1$. We begin by observing that the inclusion of categories $D^T\subset C^{\subsetneq T}$ is \emph{final}, which means that for any $C^{\subsetneq T}$-indexed diagram $F$, the natural map
\[\colim_{D^T} F\to \colim_{C^{\subsetneq T}} F\] is an isomorphism \cite[Definition~8.3.2]{Riehl}.

By the standard characterization of final functors (see Riehl~\cite[Lemma~8.3.4]{Riehl}), the inclusion $D^T\subset C^{\subsetneq T}$ is final if and only if for every object $U\in C^{\subsetneq T}$, the under-category $U/D^T$ is non-empty and connected.
In our case, $U$ is a proper subset $U\subsetneq T$, and $U/D^T$ is simply the poset of subsets $S$ such that $U\subset S\subset T$ and $\abs{T}-2\leq \abs{S}\leq \abs{T}-1$. If $U$ lies in $D^T$, it is initial in $U/D^T$, so $U/D^T$ is not just connected but contractible. Otherwise, since $U\subsetneq T$, there exists some $S_0\supset U$ with $\abs{S_0}=\abs{T}-2$, so $U/D^T$ is nonempty. For any other $S\in U/D^T$ with $\abs{S\Delta S_0}\leq 2k$, there exists a chain $S_0\subset S'_0\supset S_1\subset \cdots\supset S_k\subset S$ in $U/D^T$ with $\abs{S_i}=\abs{T}-2$ and $\abs{S'_i}=\abs{T}-1$, so $U/D^T$ is connected. Therefore $D^T\subset C^{\subsetneq T}$ is final as claimed, and so $\colim_{S\subsetneq T}V_S$ can be computed instead as $\colim_{S\in D^T} V_S$.

Consider the standard coequalizer formula for the colimit over $D^T$:
\begin{equation}
\label{eq:coequalizer}
\colim_{S\in D^T}V_S=\bigoplus_{S\in D^T} V_S\ /\ \big\langle u-f_*(u)\,\big|\,f\in \Hom_{D^T}(U,S), u\in V_U\big\rangle
\end{equation}
Since $u-\id_*u=0$, we can restrict in \eqref{eq:coequalizer} to non-identity morphisms $f\in \Hom_{D^T}(U,S)$. Such morphisms exist only when $\abs{U}=\abs{T}-2$, in which case there exist precisely two subsets $S^1_U,S^2_U$ for which there exist non-identity morphisms $i^1_U\colon U \to S^1_U$ and $i^2_U\colon U\to S^2_U$. Using the relations $u\equiv (i^2_U)_*(u)$ we can remove those $U\in D^T$ with $\abs{U}=\abs{T}-2$ from the sum \eqref{eq:coequalizer}, reducing it to
\begin{align*}
\colim_{S\in D^T}V_S &= \bigoplus_{\substack{S\subset T\\\abs{S}=\abs{T}-1}}V_S\ /\ \big\langle (i^1_U)_*(u)-(i^2_U)_*(u)\,|\, U\subset T, \abs{U}=\abs{T}-2,u\in V_U\big\rangle\\
&= (\tSminus{1}V)_T / \im(d\colon (\tSminus{2}V)_T\to (\tSminus{1}V)_T) = \coker(d\colon \tSminus{2}V\to \tSminus{1}V)_T
\end{align*}
By definition, $H_1(\tSminus{*}V)_T$ is the kernel of the map $\coker(d\colon \tSminus{2}V\to \tSminus{1}V)_T
\to V_T$  induced by $d\colon \tSminus{1}V\to \tSminus {0}V=V$. Since $d$ sends the factor $V_S$ of $\tSminus{1}V$ to $V_T$ by $i^S_*$, this induced map coincides with the universal map $\colim_{S\in D^T}V_S\to V_T$ of \eqref{eq:colimmap}. Since $\colim_{S\subsetneq T}V_S=\colim_{S\in D^T}V_S$, we conclude that $H_1(\tSminus{*}V)_T=\ker(\colim_{S\in D^T}V_S\to V_T)$ as claimed.\end{proof}

Since $T$ is terminal in the poset $C^{\subseteq T}$, the map $\colim_{S\subseteq T} V_S\to V_T$ is always an isomorphism. Therefore as a consequence of Lemmas~\ref{lem:H0identify} and \ref{lem:H1identify}, we have the following corollary.
\begin{corollary}
\label{cor:H0H1}
Let $V$ be an FI-module. Then for each finite set $T$,
\[H_0(\tSminus{*}V)_T=0\text{ and }H_1(\tSminus{*}V)_T=0\qquad\iff\qquad \colim_{S\subsetneq T}V_S=V_T.\]
\end{corollary}

\para{The homology $H_a(\tSminus{*}V)$}
The FI-module $H_0(V)$ has the property that the natural map $\X_1\colon H_0(V)\to S_{+1} H_0(V)$ {\em vanishes}; in fact, from Definition~\ref{def:H0} we see that $H_0(V)$ is the largest quotient of $V$ with this property. In particular, $H_0(V)$ is a torsion FI-module.  The main content of the following proposition is that the homology groups $H_a(\tSminus{*}V)$ enjoy the same property for every $a$.

\begin{proposition}
\label{prop:homologygraded}
Let $V$ be an FI-module. Then $H_a(\tSminus{*}V)$ is a torsion FI-module for any $a \geq 0$.  If $V$ is a finitely-generated FI-module over a Noetherian ring, we have furthermore that $H_a(\tSminus{*}V)\approxeq 0$ for each $a\geq 0$.
\end{proposition}

\begin{proof}   
We begin by proving that $H_a(\tSminus{*}V)$ is torsion; in fact, we will prove the stronger assertion that the map $\X_1\colon H_a(\tSminus{*}V)\to S_{+1}H_a(\tSminus{*}V)$ is zero for all $a\geq 0$.
The naturality of $\X_1$ implies that this map is induced by a map of FI-complexes
\beq
\X_1\colon \tSminus{*}V \to S_{+1} \tSminus{*}V.
\eeq
We will show that $X_1$ induces the zero map on homology by exhibiting an explicit chain homotopy from $X_1$ to 0.
  
If $f\colon [a] \inj S$ is an injection of finite sets, we let $\overline{f}\colon [a+1]\inj S \disjoint \set{-1}$ be the map defined 
by 
\[\overline{f}(i)=\begin{cases}-1&\text{ if }i=1\\f(i-1)&\text{ otherwise}\end{cases}\]
We then define $\widetilde{G}\colon B_a V \ra S_{+1} B_{a+1} V$ by \[\widetilde{G}\colon B_V([a],S,f)=V_{S - f([a])}\overset{=}{\longrightarrow}V_{S\disjoint [-1] - \overline{f}([a+1])}= B_V([a+1], S \disjoint \set{-1}, \overline{f}).\]   We have
\beq
d\widetilde{G}\colon B_V([a],S,f) \ra B_V([a+1], S \disjoint \set{-1}, \overline{f}) \overset{\bigoplus (-1)^i d_i}{\longrightarrow}\bigoplus_{i=1}^{a+1} B_V([a], S \disjoint \set{-1}, \overline{f} \circ s_i)
\eeq
and
\beq
\widetilde{G}d\colon B_V([a],S,f) \ra \bigoplus_{i=1}^a B_V([a-1],S,f \circ s_i) \overset{\bigoplus (-1)^i d_i}{\longrightarrow}\bigoplus_{i=1}^{a} B_V([a], S \disjoint \set{-1}, \overline{f \circ s_i})
\eeq
where the summands labeled by $i$ are twisted by $(-1)^i$ coming from $d=\sum (-1)^i d_i$.

We have the identity $\overline{f \circ s_i} = \overline{f} \circ s_{i+1}$ for $1\leq i\leq a$, so in the sum $d\widetilde{G}+\widetilde{G}d$ these terms cancel, leaving us with the map
\beq
d\widetilde{G}+\widetilde{G}d\colon B_V([a],S,f) \ra B_V([a], S \disjoint \set{-1}, \overline{f} \circ s_{1})
\eeq
But $\overline{f} \circ s_{1}\colon [a] \inj S \disjoint \set{-1}$ is just the composition $\overline{f} \circ s_{1}=i_{[-1]}\circ f$ of $f\colon [a]\inj S$ with the natural inclusion $i_{[-1]}\colon S\inj S \disjoint \set{-1}$. By definition, $X_1$ is the map induced by $i_{[-1]}$, so we conclude that \[d\widetilde{G} + \widetilde{G}d = (-1)^1 X_1=-X_1\colon B_aV\to S_{+1}B_{a+1}V.\]

Since the inclusion of $S_a$ into $S_{a+1}$ defined by $[a]\inj [a]\disjoint [-1]$ preserves $(-1)^\sigma$, the map $\widetilde{G}$ descends to a map $G\colon \tSminus{a}V\to S_{+1}\tSminus{(a+1)}V$. The computation above descends to the identity \[dG + Gd = -X_1\colon \tSminus{*}V \to S_{+1} \tSminus{*}V.\] Therefore $G$ exhibits a chain homotopy from $X_1$ to $0$ on $\tSminus{*}V$.
It follows that $X_1\colon H_a(\tSminus{*}V)\to S_{+1}H_a(\tSminus{*}V)$ is 0, and thus in particular that $H_a(\tSminus{*}V)$ is torsion.

Now suppose $V$ is a finitely-generated FI-module over a Noetherian ring.  By Lemma~\ref{lem:shiftfg} we know that $\tSminus{a}V$ is finitely generated, so Theorem~\ref{thm:noetherian} implies that its subquotient $H_a(\tSminus{*}V)$ is finitely generated as well.  Since $H_a(\tSminus{*}V)$ is torsion, this implies that $H_a(\tSminus{*}V) \approxeq 0$ by Lemma~\ref{lem:torsion}.
\end{proof}
\begin{remark}
A version of Proposition~\ref{prop:homologygraded} (with the additional assumptions that $V$ is finitely presented in some sense, and also that $V$ is an FI-module over a field whose characteristic is larger than the location of the ``relations'' of $V$) is proved by Putman in \cite[Proposition~4.5]{P}.
\end{remark}

We are now ready to prove Theorem~\ref{thm:inductive}, whose statement we recall. \begin{theoremC}
Let $V$ be a finitely-generated FI-module over a Noetherian ring $R$. Then there exists some $N\geq 0$ such that for all $n\in \N$:
\begin{equation*}
V_n=\ \colim_{\substack{S\subseteq [n]\\\abs{S}\leq N}}\  V_S
\end{equation*}
\end{theoremC}
We will prove the equivalent statement that for any finite set $T$:
\begin{equation}
\tag{\astT{T}}
\label{eq:astT}
\text{The natural map }\colim_{\substack{S\subseteq T\\\abs{S}\leq N}} V_S\to V_T \text{ is an isomorphism.}
\end{equation}
\begin{proof}[Proof of Theorem~\ref{thm:inductive}]
Under our assumptions, Proposition~\ref{prop:homologygraded} states that $H_a(\tSminus{*}V)\approxeq 0$ for all $a\geq 0$. 
In particular, we can fix some $N\geq 0$ such that $H_0(V)_n=0$ and $H_1(V)_n=0$ for all $n>N$. We will prove that for this $N$ the claim \eqref{eq:astT} holds for all finite sets $T$, by induction on $\abs{T}$. Our base case is $\abs{T}\leq N$. In this case the condition $\abs{S}\leq N$ is vacuous, and the claim \eqref{eq:astT} asserts that the natural map $\colim_{S\subseteq T} V_S\to V_T$ is an isomorphism. This is true for any $V$, since $T$ is terminal in the poset $\{S\subset T\}$. 

Fix a finite set $T$ with $\abs{T}>N$, and assume that  (\astT{U}) holds whenever $\abs{U}<\abs{T}$. 
For any map of posets $g\colon P\to Q$ and any $P$-indexed diagram $F$, it holds that \[\colim_{p\in P} F(p)=\colim_{q\in Q}  \colim_{\substack{p\in P\\g(p)\leq q}} F(p).\] Applying this to the inclusion of $\{S\subset T|\abs{S}\leq N\}$ into $\{U\subsetneq T\}$, we find that 
\begin{equation}
\label{eq:colimcolim}
\colim_{\substack{S\subseteq T\\\abs{S}\leq N}} V_S
=\colim_{U\subsetneq T}\colim_{\substack{S\subseteq U\\\abs{S}\leq N}}V_S.
\end{equation} Applying the inductive assumption $(\astT{U})$ gives $\colim_{\substack{S\subset U\\\abs{S}\leq N}}V_S=V_U$ for each $U\subsetneq T$. Therefore \eqref{eq:colimcolim} simplifies to $\colim_{U\subsetneq T}V_U$.
Since $\abs{T}>N$ we have $H_0(V)_T=H_1(V)_T=0$, so Corollary~\ref{cor:H0H1} states that $\colim_{U\subsetneq T}V_U=V_T$. Summing up, we have
\[\colim_{\substack{S\subseteq T\\\abs{S}\leq N}} V_S
=\colim_{U\subsetneq T}\colim_{\substack{S\subseteq U\\\abs{S}\leq N}}V_S
=\colim_{U\subsetneq T}V_U=V_T.\] This concludes the proof of \eqref{eq:astT}.
\end{proof}

\section{Congruence FI-groups (Proof of Theorem~D)}
\label{sec:congruence}
In this section we will prove Theorem~\ref{thm:congruencefg}, on the homology of the congruence FI-group $\Gamma_\bullet(\p)$.

\para{The congruence FI-group $\Gamma_\bullet(\p)$} Given a commutative ring $R$, let $M(1)=M(1)_{/R}$ denote the FI-module taking a finite set $S$ to the free $R$-module with basis $\{e_s|s\in S\}$, and let $M(1)^*$ denote the FI-module taking a finite set $S$ to $\Hom_R(M(1)_S,R)$. 
Their tensor product is the endomorphism FI-algebra $\End M(1)=M(1)\otimes M(1)^*$, and the invertible endomorphisms form the FI-group $\GL(M(1))$; this definition agrees with the FI-group $\GL_\bullet(R)$ defined by \eqref{eq:GLFIgroup} in the introduction.

\begin{remark}
\label{rem:M1M1dual}
There is an isomorphism of FI-modules from $M(1)$ to $M(1)^*$ which sends $e_s$ to the functional $\lambda_s\colon M(1)_S\to R$ defined by $\lambda_s(e_t)=\delta_{st}$. Nevertheless, we maintain the distinction because the natural actions of $\GL(M(1))$ on $M(1)$ and on $M(1)^*$ are not equivalent. Taking $S=[n]$, we have canonical isomorphisms $M(1)_n\simeq M(1)^*_n\simeq R^n$ and $\GL(M(1))_n\simeq \GL_n(R)$; the action on $M(1)$ is by the standard representation of $\GL_n(R)$ on $R^n$, while the action on $M(1)^*$ is by the dual representation $g\mapsto (g^{-1})^\top$.
\end{remark}

For any ideal $\p\subset R$, the natural reduction map from $R$ to $\FF\coloneq R/\p$ induces maps $M(1)_{/R} \to M(1)_{/\FF}$ and $\GL(M(1)_{/R})\to \GL(M(1)_{/\FF})$.
As in the introduction, the {\em congruence FI-group $\Gamma_\bullet(\p)$} is defined by the short exact sequence of FI-groups:
\[1\to \Gamma_\bullet(\p)\to \GL(M(1)_{/R})\to \GL(M(1)_{/\FF})\]
\begin{proof}[Proof of Theorem~\ref{thm:congruencefg}]
Fix a number field $K$ with ring of integers $\O_K$, and let $\p\subsetneq \O_K$ be a proper ideal. Fix also a Noetherian ring $A$, and consider the FI-module $\HH_m\coloneq \HH_q(\Gamma_\bullet(\p);A)$ over $A$; our goal is to prove that $\HH_m$ is finitely generated.

We work with a more naive version of the complex used by Putman in \cite{P}. Consider $M(1)\times M(1)^*=M(1)_{/\O_K}\times M(1)_{/\O_K}^*$ as an FI-set (ignoring any additive strucure).  Our complex $X_\bullet$ will be an FI-simplicial complex with vertex set contained in $M(1)\times M(1)^*$.   

Consider the FI-simplicial complex $\Delta^{\bullet-1}$ which assigns to any finite set $S$ the full simplicial complex $\Delta^{\bullet-1}(S)$ with vertex set $S$.   Thus $\Delta^{\bullet-1}(n)$ is the standard $(n-1)$-simplex $\Delta^{n-1}$, its FI-endomorphisms act by the standard action of $S_n$ on $\Delta^{n-1}$, and any injective map $S\hookrightarrow T$ of sets induces a simplicial inclusion 
$\Delta^{\bullet-1}(S)\to\Delta^{\bullet-1}(T)$.  See \cite[Example 2.11]{CEF} for more on this FI-simplicial complex.

Let $D_\bullet$ denote the FI-simplex $\Delta^{\bullet-1}$,  considered as embedded in $M(1)\times M(1)^*$ as the full simplex on the elements 
$\{(e_s, \lambda_s)|s\in S\} \subset M(1)_S\times M(1)_S^*$. We define the FI-simplicial complex $X_\bullet$ to be
\[X_\bullet\coloneq \Gamma_\bullet(\p)\cdot D_\bullet.\] In other words, $X_\bullet$ is the simplicial complex with vertex set contained in  $M(1)\times M(1)^*$ consisting of all of those simplices lying in the $\Gamma_\bullet(\p)$-orbit of $D_\bullet$.

No element of $\Gamma_\bullet(\p)$ takes any simplex of $D_\bullet$ to a different simplex of $D_\bullet$, as these simplices are distinguished by their reduction in $M(1)_{/\FF}\times M(1)_{/\FF}^*$, which is preserved by the action of $\Gamma_\bullet(\p)$. (This is where we use that $\p$ is a \emph{proper} ideal of $\O_K$.) Thus $D_\bullet$ is by definition a fundamental domain for the action of $\Gamma_\bullet(\p)$ on $X_\bullet$, and we have a canonical identification
\[X_\bullet/\Gamma_\bullet(\p)\simeq D_\bullet.\]

From such an action we  obtain in the usual way (see \cite[Equation~VII.7.2]{B}) a spectral sequence converging to the equivariant homology $H_*^{\Gamma_\bullet(\p)}(X_\bullet)$. Although our complex $X_\bullet$ differs from the complex $\mathcal{SB}_n(\O_K,\p)$ considered by Putman, he notes in \cite[Lemma~3.2]{P} that $X_\bullet \approxeq \mathcal{SB}_\bullet(\O_K,\p)$ since $\O_K$ satisfies Bass's stable range condition $S_3$. In particular, Putman deduces from Charney \cite[Theorem~3.5]{Charney} that the complex $\mathcal{SB}_n(\O_K,\p)$ is $(\frac{n}{2}-2)$-acyclic \cite[Lemma~3.1]{P}, so we have $\widetilde{H}_m(X_\bullet)\approxeq 0$ for all $m\geq 0$. This implies (see \cite[Proposition~VII.7.3]{B}) that \[H_m^{\Gamma_\bullet(\p)}(X_\bullet)\approxeq H_m(\Gamma_\bullet(\p))=\HH_m;\] in other words, the equivariant homology computed by the spectral sequence is asymptotically identical with the ordinary homology FI-module $\HH_m$ that is our object of study here.

Let us consider the $E^1$ page of this spectral sequence more closely. Since $D_\bullet$ is a fundamental domain for the action, we have (see \cite[Equation~VII.7.7]{B}):
\begin{equation}
\label{eq:E1first}
E^1_{pq}=\bigoplus_{\substack{\sigma \text{ a $p$-simplex}\\\text{of }D_\bullet}}H_q(\Stab_{\Gamma_\bullet(\p)}(\sigma);R)\qquad\implies\  H_m^{\Gamma_\bullet(\p)}(X_\bullet)\approxeq\HH_{p+q}
\end{equation}

Each $p$-simplex $\sigma$ of $D_S$ is the full simplex on $\{(e_u, \lambda_u)|u\in U\} \subset M(1)_S\times M(1)_S^*$ for some $U\subset S$ with $\abs{U}=p+1$.   Let $T=S-U$.  The FI-group structure on $\Gamma_\bullet(\p)$ yields an inclusion $\Gamma_T(\p) \inj \Gamma_S(\p)$, and the stabilizer in $\Gamma_S(\p)$ of the simplex $\sigma_U$ is precisely the subgroup $\Gamma_T(\p)$.   This shows that \[(E^1_{pq})_S=\bigoplus_{\substack{T\subset S\\\abs{T}=\abs{S}-p-1}}H_q(\Gamma_T(\p);R).\] Since a permutation of $U=S-T$ acts on the orientation of the $p$-simplex $\sigma_U$ according to its sign, comparing with Definition~\ref{def:negshift}, we see that we can identify $E^1_{pq}$ with $\tSminus{p-1}(\HH_q)$; more than this, we can identify the $q$th row ($E^1_{\ast,q},d^1)$ with the complex 
$\tSminus{*-1}\HH_q$ 
 from \eqref{eq:shiftcomplex}, excluding the last term $\tSminus{0}\HH_q=\HH_q$.
We have in particular $E^1_{0,m}=\tSminus{1}\HH_m$, and the edge map $E^1_{0,m}\to H_{m}^{\Gamma_\bullet(\p)}(X_\bullet)$ factors as
\begin{equation} \label{eq:abut}
\tSminus{1}(\HH_m)=E^1_{0,m}\twoheadrightarrow E^\infty_{0,m}\hookrightarrow H_{m}^{\Gamma_\bullet(\p)}(X_\bullet)\approxeq \HH_m.
\end{equation}
The composition of these maps is just the boundary map $\tSminus{1}(\HH_m) \ra \HH_m$ appearing in \eqref{eq:shiftcomplex}.

We now prove by induction on $m$ that $\HH_m$ is a finitely-generated FI-module. For the base case, we have $\HH_0=M(0)$, which is finitely generated by definition.

Suppose we know that $\HH_{q}$ is finitely generated for all $q < m$. The cokernel of the map $E^\infty_{0,m}\to H_{m}^{\Gamma_\bullet(\p)}(X_\bullet)$ has a filtration whose graded quotients are isomorphic to $E^\infty_{p,m-p}$ for $1\leq p \leq m$. Since $E^2_{pq}=H_p(\tSminus{*}\HH_q)$ and $R$ is Noetherian, Proposition~\ref{prop:homologygraded} tells us that $E^2_{pq} \approxeq 0$ for all $p\geq 0$ and all $q < m$. Since $E^\infty_{pq}$ is a subquotient of $E^2_{pq}$, it follows that $E^\infty_{pq} \approxeq 0$ for all $p\geq 0$ and all $q < m$. This shows that $\coker(E^\infty_{0,m} \ra H_{m}^{\Gamma_\bullet(\p)}(X_\bullet))\approx 0$.  Via \eqref{eq:abut}, this implies that $\coker(\tSminus{1}(\HH_m)\to \HH_m) \approxeq 0$.
By Lemma~\ref{lem:H0identify} this says that $H_0(\HH_m)\approxeq 0$, which by Lemma~\ref{lem:cokerfg}.1 is equivalent to saying that $\HH_m$ is generated in finite degree.

 The existence of the Borel--Serre compactification \cite{BS} implies that $H_m(\Gamma_n(\p);R)$ is a finitely-generated $R$-module for all $m\geq 0$ and all $n\geq 0$. Therefore by Lemma~\ref{lem:cokerfg}.2, $\HH_m$ is generated in finite degree if and only if $\HH_m$ is finitely generated. This shows that $\HH_m$ is finitely generated, completing the inductive step of the proof.
\end{proof}

\begin{remark}  The homology groups $H_m(\Gamma_n(\p);\Z)$ do not merely carry an action of $S_n$, but of the larger linear group $\SL_n(\FF)$, in which $S_n$ is contained as a subgroup. (This uses the nontrivial result that the mod-$\p$ reduction $\SL_n\O_K\to \SL_n\FF$ is actually surjective.) In keeping with the philosophy of \cite[\S8]{CF} one might ask whether the groups $H_m(\Gamma_n(\p);\Z)$ obey an appropriate notion of ``representation stability" with respect to the action of the family $\{\SL_n\FF\}$.
\end{remark}

\section{Configuration co-FI-spaces (Proof of Theorem~\ref{thm:configurationsfg})}
Let $R$ be a Noetherian ring and let $M$ be a connected orientable manifold of dimension $\geq 2$ with $H^*(M;R)$ finitely generated.
We recall from the introduction that $\Conf(M)$ is the co-FI-space sending a finite set $S$ to the space $\Inj(S,M)$ of injections of $S$ into $M$. Let $M^\bullet$ be the co-FI-space defined by $M^S=\Map(S,M)$.  There is a natural inclusion  $i\colon \Conf(M)\inj M^\bullet$ as co-FI-spaces. 

\begin{lemma}
\label{lemma:cohoMbullet}
Let $M$ be a connected space with the homotopy type of a CW complex with finitely many cells in each dimension.  Then for all $m\geq 0$, the FI-module $H^m(M^\bullet;R)$ is generated in finite degree.
\end{lemma}

\begin{proof} When $R$ is a field $k$, the lemma can be deduced without difficulty from the K\"{u}nneth theorem and the results of \cite{CEF}, since $H^*(M^\bullet;k)=H^*(M;k)^{\otimes\bullet}$. However for general $R$, the relation between the cohomology of $M^n$ and that of $M$ is more complicated; we handle this by working directly at the level of cochains.

We have assumed that $M$ is homotopy equivalent to a CW complex; since $M$ is connected, we may assume that this CW complex has only a single $0$-cell. Let $\CC_*$ be the corresponding cellular chain complex over $R$; this is a bounded-below chain complex of finitely generated projective $R$-modules with $\CC_0=R$.

We recall from \cite[Definition 2.71]{CEF} the definition of the co-FI-chain complex $\CC_*^{\otimes \bullet}$. By definition, in degree $n$ it is $(\CC_*^{\otimes \bullet})_{[n]}\coloneq\CC_*^{\otimes n}$, which is a bounded-below chain complex of finitely generated projective $R$-modules. An injection $f\colon [n]\hookrightarrow [m]$ induces the map $f^*\colon \CC_*^{\otimes m}\hookrightarrow \CC_*^{\otimes n}$ which on each factor lying in $[m]-f([n])$ is the projection onto $\CC_0=R$, and permutes the remaining  factors according to $f^{-1}$ (with appropriate sign based on the grading).

The Eilenberg--Zilber theorem states that the singular chain complex $C_*(M^n)$ of $M^n$ is quasi-isomorphic to the $n$-fold derived tensor product $C_*(M)^{\tensor^{\mathbb{L}} n}$. Since $\CC_*$ is quasi-isomorphic to $C_*(M)$ we have $C_*(M)^{\tensor^{\mathbb{L}} n}=(\CC_*)^{\tensor^{\mathbb{L}} n}$. But $(\CC_*)^{\tensor^{\mathbb{L}} n}=(\CC_*)^{\tensor n}$, since $\CC_*$ is a complex of projective $R$-modules. Therefore $(\CC_*)^{\otimes n}$ is quasi-isomorphic to $C_*(M^n)$. In other words, since $\CC_*$ consists of projective modules and coincides with $C_*(M)$ in the derived category $D^b(R)$, the co-FI-chain complexes $(\CC_*)^{\tensor \bullet}$ and $C_*(M^\bullet)$ define the same co-FI-object of $D^b(R)$.

In particular, the cohomology $H^*(M^\bullet;R)=\Ext^*(C_*(M^\bullet),R)$ can be computed as the cohomology of the complex $\Hom((\CC_*)^{\tensor \bullet},R)$, which is now an FI-chain complex of finitely generated projective $R$-modules.
Denote the piece of this complex in grading $m$ by $\Hom((\CC_*)^{\tensor \bullet},R)^{m}$:
\[\Hom((\CC_*)^{\tensor n},R)^m=\bigoplus_{m_1+\cdots+m_n=m} \Hom(\CC_{m_1}\otimes\cdots\otimes \CC_{m_n},R)\]  When $n>m$ every such factor must have $m_i=0$ for some $i$, and thus lies in the image of $f_*$ for some $f\colon [n-1]\hookrightarrow [n]$. Therefore the FI-module $\Hom((\CC_*)^{\tensor \bullet},R)^{m}$ is finitely generated in degree $m$.
Since $H^m(M^\bullet)$ is a subquotient of this finitely generated FI-module, it is finitely generated by Theorem~\ref{thm:noetherian}.
\end{proof}

\begin{proof}[Proof of Theorem~\ref{thm:configurationsfg}]
We consider the inclusion of co-FI-spaces $i\colon \Conf(M)\inj M^\bullet$, and the resulting  Leray spectral sequence  of FI-modules over $R$: \[E_2^{p,q}=H^p(M^\bullet;R^qi_*(\underline{R}))\implies H^{p+q}(\Conf(M);R)\] Our first goal is to verify that $E_2^{p,q}$ is finitely generated as an FI-module for each $p,q\geq 0$.   
Over $\Q$ this argument was given in the proof of \cite[Theorem~4.1]{CEF}, and the same outline works here; the main difference over a general Noetherian ring $R$ was in Lemma~\ref{lemma:cohoMbullet}.

Totaro describes the $E_2$ page of this spectral sequence \cite[Theorem 1]{To}, and in particular he shows that $E_2^{*,*}$ is generated by the subalgebras $E_2^{*,0}$ and $E_2^{0,*}$ (see the proof of \cite[Theorem~4.1]{CEF} for more details). The former is isomorphic to $H^*(M^\bullet;R)$, which is finitely generated by Lemma~\ref{lemma:cohoMbullet}.

Totaro proves that the subalgebra $E_2^{0,*}$ is generated by $E_2^{0,d-1}$, which is generated in degree 2 (by the element ``$G_{12}$'', in Totaro's notation). Since this is a first-quadrant spectral sequence, only finitely many terms along each axis can multiply to any given entry. Each entry $E_2^{p,q}$ is thus the quotient of a finite direct sum of finite tensor products of finitely-generated FI-modules. By the basic proposition \cite[Proposition~2.61]{CEF}, such a finite tensor product is itself finitely generated. It follows that $E_2^{p,q}$ is finitely generated as well.

Since $E_\infty^{p,q}$ is a subquotient of $E_2^{p,q}$, Theorem~\ref{thm:noetherian} implies that $E_{\infty}^{p,q}$ is finitely generated for each $p\geq 0$ and $q\geq 0$. The cohomology FI-module $H^m(\Conf(M);R)$ has a finite-length filtration whose graded quotients are of this form, so by \cite[Proposition~2.17]{CEF} the FI-module $H^m(\Conf(M);R)$ is itself finitely generated, as desired.
\end{proof}

\section{Coinvariant co-FI-algebras (Proof of Theorem~\ref{thm:coinvariantsfg})}
Fix a commutative Noetherian ring $A$, and fix an integer $r\geq 1$. We recall from the introduction that $A[\XX^{(r)}]$ is the $\Z_{\geq 0}^r$-graded co-FI-algebra which sends a finite set $S$ to the free commutative $A$-algebra on generators indexed by $[r]\times S$. Its quotient by the ideal of $\Aut(S)$-invariant polynomials with zero constant term defines the $\Z_{\geq 0}^r$-graded co-FI-algebra $R^{(r)}$, the \emph{$r$-diagonal coinvariant co-FI-algebra}. 

\begin{proof}[Proof of Theorem~\ref{thm:coinvariantsfg}]
The co-FI-algebra $A[\XX^{(r)}]$ is the free commutative $A$-algebra generated by the co-FI-module $M(1)^\vee$, so the FI-algebra $A[\XX^{(r)}]^\vee$ is the free commutative $A$-algebra generated by the FI-module $M(1)$. In particular, if $J=(j_1,\ldots,j_r)$, the graded piece $A[\XX^{(r)}]^\vee_J$ is isomorphic to $\Sym^{j_1}M(1)\otimes \cdots \otimes \Sym^{j_r}M(1)$. This is a quotient of $M(1)^{\tensor \abs{J}}$, which is finitely generated by \cite[Proposition~2.61]{CEF}, so $A[\XX^{(r)}]^\vee_J$ is a finitely generated FI-module over $A$.

Since  $R^{(r)}_J$ is a quotient of $A[\XX^{(r)}]_J$, its dual $(R^{(r)}_J)^\vee$ naturally embeds as a sub-FI-module of $A[\XX^{(r)}]_J^\vee$.  Theorem~\ref{thm:noetherian} therefore implies that  $(R^{(r)}_J)^\vee$ is finitely generated as desired.
\end{proof}

\small
\noindent
\begin{tabular}{lll}
Dept. of Mathematics\ \ \ \ \ \ \ \ \ \ \ \ \ \ \ &
Dept. of Mathematics\ \ \ \ \ \ \ \ \ \ \ \ \ \ \ &
Dept. of Mathematics\\
Stanford University&
University of Wisconsin&
University of Chicago\\
450 Serra Mall&
480 Lincoln Drive&
5734 S. University Ave.\\
Stanford, CA 94305&
Madison, WI 53706&
Chicago, IL 60637\\
\myemail{church@math.stanford.edu}&
\myemail{ellenber@math.wisc.edu}&
\myemail{farb@math.uchicago.edu}\\
&\myemail{nagpal@math.wisc.edu}&
\end{tabular}

%
%
%
%
%
%
\end{document}